\documentclass{article}
\usepackage[numbers,sort&compress]{natbib}
\usepackage[utf8]{inputenc}
\usepackage[T1]{fontenc}
\usepackage{fullpage}
\usepackage{hyperref}
\usepackage{url}
\usepackage{booktabs}
\usepackage{amsfonts}
\usepackage{nicefrac}
\usepackage{microtype}
\usepackage{xcolor}
\usepackage{graphicx}
\usepackage{amsmath,amsfonts,bm}

\def\1{\bm{1}}

\DeclareMathAlphabet{\mathsfit}{\encodingdefault}{\sfdefault}{m}{sl}
\SetMathAlphabet{\mathsfit}{bold}{\encodingdefault}{\sfdefault}{bx}{n}

\usepackage{multirow}
\usepackage{algorithm}
\usepackage{algorithmic}
\usepackage{adjustbox}
\usepackage[numbers]{natbib}
\usepackage{bbm}
\usepackage{bm}
\usepackage{mathrsfs}
\usepackage{amsmath} %数学公式
\usepackage{float}
\usepackage{diagbox}
\usepackage{subcaption}
\usepackage{amssymb}
\newtheorem{thm}{Theorem}[section]

\newtheorem{lemma}[thm]{Lemma}
\newtheorem{assume}{Assumption}
\newtheorem{remark}[thm]{Remark}

\newtheorem{prop}[thm]{Proposition}

\newenvironment{proof}{Proof:}{\hfill$\square$}
\usepackage{threeparttable}

\usepackage{makecell}
\usepackage{enumerate}
\usepackage[figuresright]{rotating}
\usepackage{setspace}
\usepackage{colortbl}
\definecolor{bgcolor}{rgb}{0.8,1,1}
\definecolor{bgcolor2}{rgb}{0.8,1,0.8}

\usepackage{tikz}
\usetikzlibrary{positioning, shapes.geometric}

\newcommand{\norm}[1]{\left\|#1\right\|}

\def\sqr#1#2{{\vcenter{\vbox{\hrule height.#2pt
				\hbox{\vrule width.#2pt height#1pt \kern#2pt
					\vrule width.#2pt}
				\hrule height.#2pt}}}}
\def\approxleq{ \kern3pt \mbox{\raisebox{.6ex}{$<$}} \kern-8pt
	\mbox{\raisebox{-.6ex}{$\sim$}} \kern5pt}

\def\norm#1{\|#1 \|}
\def\inprod#1#2{\langle#1,\,#2 \rangle}

   \def\cT{{\cal T}}
    \def\cG{{\cal G}} 
   
\def\cS{{\cal S}}   \def\cF{{\cal F}}  
\def\cI{{\cal I}} \def\cX{{\cal X}} \def\cY{{\cal Y}} \def\cZ{{\cal Z}}
  \def\cL{{\cal L}}

\def\blkdiag{\textrm{\bf blkdiag}}

\def\ni{\noindent}

\def\bx{\bar{x}}
\def\bz{\bar{z}}
\def\by{\bar{y}}

\newlength{\len}
\title{An Accelerated Proximal Alternating Direction Method of Multiplier for Optimal Decentralized Control of Uncertain Systems}

\author{Bo Yang%\thanks{Equal Contribution.}
    \thanks{Department of Mathematics,
        Beijing University of Technology, beijing, China (\texttt{bbo\_yang@163.com}).}
    \and
    Xinyuan Zhao\thanks{Department of Mathematics,
        Beijing University of Technology, beijing, China (\texttt{xyzhao@bjut.edu.cn}). The research of Xinyuan Zhao is supported in part by the National Natural Science Foundation of China 12271015.}
    \and
    Xudong Li\thanks{School of Data Science, Fudan University, shanghai, China (\texttt{lixudong@fudan.edu.cn}). The research of Xudong Li is supported in part by the National Natural Science Foundation of China 12271107.}	
    \and
    Defeng Sun\thanks{Department of Applied Mathematics, The Hong Kong Polytechnic University, Hong Kong, China (\texttt{defeng.sun@polyu.edu.hk}). The research of Defeng Sun is supported in part by the Hong Kong Research Grant Council Grant  N\_PolyU504/19.}
}

\begin{document}
\maketitle

\begin{abstract}
To ensure the system stability of the $\mathcal{H}_{2}$-guaranteed cost optimal decentralized control problem (ODC), we formulate an approximate semidefinite programming (SDP) problem based on the block diagonal structure of the gain matrix of the decentralized controller. To reduce data storage and improve computational efficiency, we apply the property of the Kronecker product to vectorize the SDP problem into a conic programming (CP) problem. Then, a proximal alternating direction method of multipliers (PADMM) is proposed to solve the dual of the resulted CP. By establishing the equivalence between the semi-proximal ADMM and the (partial) proximal point algorithm, we find the non-expansive operator of PADMM so that we can use Halpern fixed-point iteration to accelerate the proposed algorithm. Finally, we verify that the sequence generated by the proposed accelerated PADMM has a fast convergence rate for the Karush-Kuhn-Tucker residual. Through numerical experiments, it has been demonstrated that the accelerated algorithm surpasses the renowned COSMO, MOSEK, and SCS solvers in efficiently solving large-scale CP problems, particularly those arising from $\mathcal{H}_{2}$-guaranteed cost ODC problems.
\end{abstract}

\section{Introduction}\label{sec:introduction}
Numerous complex real-world systems, such as aircraft formations, automated highways, and power systems, consist of a vast number of interconnected subsystems. In these interconnected systems, the controller can only access the status information of each subsystem. To reduce the computational and communication complexity of the overall controller, the optimal decentralized control (ODC) problem with structural constraints is considered to stabilize the system and seek optimal performance. Due to its significance, the ODC problem has attracted research attention since the late 1970s \cite{Siljak11, Lunze92}.

\quad The design of a centralized controller without structural constraints for the linear quadratic regulator, linear quadratic Gaussian, $\mathcal{H}_{2}$, and $\mathcal{H}_{\infty}$ control problems typically involves solving algebraic Riccati equations. However, extending this technique to decentralized control situations is challenging due to the specific sparsity constraints inherent in the gain matrix of the decentralized controller. Furthermore, it has been known that designing a globally optimal decentralized controller is generally an NP-hard problem \cite{Witsen1968, Tsitsi1984}. Great effort has been devoted to investigating this complex problem for special types of systems, including spatially distributed systems \cite{Jadbabaie}, dynamic decoupling systems \cite{Keviczky}, weakly connection systems \cite{Siljak}, and strongly connected systems \cite{Lavaei}. Early efforts focused on designing methods based on parameterization techniques \cite{GEROMEL19941565, Chow}, which were then evolved into matrix optimization methods \cite{Scorletti, Ikeda}. While these methods have their merits, they are computationally demanding and do not always ensure global optimality. Moreover, they can be particularly resource-intensive for large-scale systems and susceptible to numerical instability. Given these limitations, it is evident that an alternative approach is necessary for designing decentralized controllers.

\quad Influenced by recent achievements of convex optimization, the focus of control synthesis problems has shifted towards finding a convex formulation that can be efficiently solved. Various convex relaxation techniques, such as those based on linear matrix inequality, semidefinite programming (SDP) \cite{Todd}, and second-order cone programming \cite{Ling}, have gained popularity in addressing the ODC problems. In addition to these convex relaxation methods, the quadratic invariance (QI) is shown to be a necessary and sufficient condition for the set of Youla parameters to be convex \cite{Lessard}, and guarantees the existence of a sparse controller \cite{Rotkowitz, SLal, Bitar,8727508}. Recently, the concept of sparse invariance was introduced by \cite{Kamgarpour}, which can be applied to design optimally distributed controllers as a method to overcome the QI restrictions. However, these invariance conditions are computationally expensive, if feasible, to verify for practical applications.
To attain high computational effectiveness, more recently, by adapting the inexact sGS decomposition technique developed in \cite{Lischur, Lisgs} for solving multi-block convex problems, the sGS semi-proximal augmented Lagrangian method was first used by \cite{JMa, 9589405} to solve conic programming (CP) relaxation of the $\mathcal{H}_{2}$-guaranteed cost ODC problem. Meanwhile, the designed method can suitably guarantee a decentralized structure, robust stability, and robust performance. However, the computational efficiency is still quite inadequate, and one possible reason is perhaps due to the fact that the algorithm must compute the projections onto the semidefinite cones twice at each iteration. In light of this, there is a high possibility to design a more efficient algorithm for solving the $\mathcal{H}_{2}$-guaranteed cost ODC problem.

\quad Acceleration techniques play a crucial role in enhancing algorithmic efficiency, attracting significant attention in optimization and related fields. Inspired by Nesterov’s accelerated gradient method \cite{Nester269}, Kim \cite{Kim2021} introduced an accelerated scheme of the proximal point algorithm (PPA) for solving the maximally monotone inclusion problem. This accelerated scheme exhibits a rapid $\mathcal{O}(1/k)$ convergence rate for the fixed-point residual of the corresponding maximally monotone operator, where $k$ denotes the number of iterations. Note that solving the maximal monotone inclusion problem can be restated equivalently as finding the fixed-point of the non-expansive operator \cite{BC17}. A more general Halpern iteration was introduced earlier \cite{Halpern} for approximating the fixed-point of the non-expansive operator. Specifically, consider a non-expansive mapping $T$ on the Hilbert space $\mathcal{H}$ and a fixed $x_{0}\in \mathcal{H}$. The Halpern fixed-point iteration takes the following form:
\begin{equation*}
x_{k+1}:=\lambda_{k}x_{0}+(1-\lambda_{k})T(x_{k}),
\end{equation*}
where the stepsizes $\lambda_k\in (0,1)$. In the Hilbert space, Lieder \cite{Lied21} recently established that the Halpern iteration with stepsizes $\lambda_k=1/(k+2)$ achieves an accelerated rate of $\mathcal{O}(1/k)$ on the norm of the fixed-point residual. Notably, the connections between the Halpern iteration with stepsizes $\lambda_k=1/(k+2)$ \cite{Halpern,Lied21} and Kim's acceleration \cite{Kim2021} were highlighted in \cite[Proposition 5]{Contreras}. Furthermore, Contreras and Cominetti in \cite[Proposition 2]{Contreras} demonstrated that the best possible rate for general Mann iterations, including the Halpern iteration, in normed spaces is bounded from below by $\mathcal{O}(1/k)$. Thus, it is natural to employ the Halpern iteration or Kim's acceleration to accelerate the convergence rate of the fixed-point residual. Along this line, an efficient Halpern-Peaceman-Rachford algorithm for solving the two-block convex programming problems including the Wasserstein barycenter problem was introduced in \cite{Zhang2022AnEH}. This algorithm achieved an appealing $\mathcal{O}(1/k)$ non-ergodic convergence rate with respect to the Karush-Kuhn-Tucker (KKT) residual. Additionally, Kim \cite{Kim2021} also proposed an accelerated alternating direction method of multipliers (ADMM) and proved an $\mathcal{O}(1/k)$ convergence rate with respect to the primal feasibility only. The convergence of these accelerated methods depends on specific assumptions, such as the full column rank of the coefficient matrix in the constraints or the strong convexity of the objective functions \cite{DRA92}. However, it's important to note that the CP problem stemming from the $\mathcal{H}_{2}$-guaranteed cost ODC problem may not necessarily satisfy these assumptions. To tackle this issue, we present an innovative accelerated proximal ADMM (PADMM) based on the Halpern iteration for solving the CP problem. Specifically, proximal terms are integrated to ensure the existence of optimal solutions to the corresponding subproblems, and the Halpern iteration \cite{Lied21} is employed as an acceleration strategy. In comparison to the non-ergodic $\mathcal{O}(1/\sqrt{k})$ iteration complexity of the majorized ADMM presented in \cite{Cui}, this paper demonstrates that the proposed accelerated PADMM achieves a faster $\mathcal{O}(1/k)$ convergence rate for the fixed-point KKT residual.

\quad In this paper, employing a parameterized approach, we approximate the ODC problem with an SDP problem that ensures the block-diagonal structure of the feedback matrix. Additionally, the stability of the system with parameter uncertainties is guaranteed \cite{JMa}. To enhance the algorithm's efficiency and minimize matrix operations, we transform the original SDP problem into an equivalent vector form, and the optimal solutions are efficiently obtained by using the accelerated PADMM.

\subsection{Main Contribution}
The contributions made in this paper to the current literature can be summarized as follows:
\begin{itemize}
\item [i)] An equivalence between the semi-proximal ADMM and the (partial) PPA is established. Utilizing this equivalence, we propose an accelerated PADMM by applying the Halpern iteration to the (partial) PPA. The proposed accelerated PADMM demonstrates a convergence rate of $\mathcal{O}(1/k)$ with respect to the norm of the KKT residual.
    
\item[ii)] By leveraging the sparsity of the problem data, we introduce a lifting technique to solve the linear systems of subproblems in the (accelerated) PADMM. Consequently, the proposed accelerated PADMM can efficiently address the large-scale CP problem originating from the ODC problem.

\item[iii)]We implement the accelerated PADMM in C language (using the MSVC compiler) and compare its performance to COSMO \cite{GarCOSMO}, MOSEK \cite{mosek}, and SCS \cite{Parikh} in solving small-medium and large-scale problems. Extensive numerical results demonstrate the efficiency and robustness of the proposed algorithm.
\end{itemize}
The remainder of this paper is organized as follows: In Section \ref{sec2}, we first introduce the generic form of the $\mathcal{H}_{2}$-guaranteed cost ODC for uncertain systems. Subsequently, we relax the $\mathcal{H}_{2}$-guaranteed cost ODC problem to the CP problem. In Section \ref{sec3}, we establish the equivalence between the semi-proximal ADMM and the (partial) PPA. Based on this equivalence, we propose an accelerated PADMM for solving the CP problem. Moreover, we provide a fast implementation of the proposed algorithm for solving the CP in Section \ref{sec4}. Section \ref{sec5} presents extensive numerical experiments demonstrating the superiority of our algorithm compared to others. Finally, we conclude the paper in Section \ref{sec6}.

\vspace{4pt}
\noindent\textbf{Notation.} The notation used in this paper is defined as follows:
\begin{enumerate}
\item[$\bullet$] Let $\mathbb{R}^{n}$ denote the $n$-dimensional real space. The identity operator, denoted $\mathcal{I}$, is the operator on $\mathbb{R}^{n}$ such that $\mathcal{I} v = v$, $\forall \, v \in \mathbb{R}^{n}$.
 %  denote identity operator equipped with inner product $\langle \cdot,\cdot\rangle,$ induced norm $\|\cdot\|,$ 
\item[$\bullet$] $\mathbb{S}^{n}$ is the set of all $n\times n$ real symmetric matrices; {$\mathbb{S}_{++}^{n}~(\mathbb{S}_{+}^{n})$} is the cone of positive (semidefinite) matrices in $\mathbb{S}^{n}$ with the trace inner product $\langle \cdot,  \cdot \rangle$ and the Frobenius norm $\| \cdot \|$. We may sometimes write $X\succ0$ $(X\succeq0)$ to indicate that $X\in \mathbb{S}_{++}^{n}~(\mathbb{S}_{+}^{n})$.   We use $\otimes$ to denote the Kronecker product of matrices. 

\item[$\bullet$] Let $\blkdiag \left\{A_{1},~A_{2},~\ldots,~A_{m}\right\}$ denote the  block diagonal matrix with diagonal entries $A_{1},~A_{2},~\ldots,~A_{m}$ and $\textbf{vec}(A)$ denote the column vector formed by stacking columns of $A$ one by one.

\item[$\bullet$] $e_{k}\in\mathbb{R}^{p\times 1}$ means that the $k$th element is one and all other elements are zeros.

\item[$\bullet$] For the set $\mathcal{X} \subseteq \mathbb{R}^{n},$ $\delta_{\mathcal{X}}$ is the indicator function for $\mathcal{X}$, i.e., $\delta_{\mathcal{X}}=0$ if $x\in \mathcal{X}$ and $\delta_{\mathcal{X}}=+\infty$ if $x\in \mathcal{X}.$ For a closed convex set $\mathcal{D}$, the Euclidean projector onto $\mathcal{D}$ is defined by $\Pi_{\mathcal{D}}(x):=\underset{s\in \mathcal{D}}{\mathrm{argmin}}\{\|s-x\|\}.$

\item[$\bullet$] The set of fixed points of an operator ${\cal T}: \mathcal{X}\rightarrow \mathcal{X}$ is denoted by  $\text{Fix}({\cal T})$, i.e.,
$\text{Fix}({\cal T}):=\{x\in \mathcal{X} \, |\, x={\cal T}(x)\}.$
\end{enumerate}

\section{Optimal decentralized control problem}\label{sec2}

In this section, we first introduce the ODC problem. Subsequently, we present an approximated CP model for the ODC problem.
\subsection{Problem statement}
Consider a linear time-invariant (LTI) system with additive disturbance in a compact form as follows:
\begin{equation}\label{a1}
\begin{aligned}
\dot{x}(t)&=Ax(t)+B_{2}u(t)+B_{1}w(t),\\
z(t)&=Cx(t)+Du(t)
\end{aligned}   
\end{equation}
with a static state feedback controller 
\begin{equation}\label{a2}
u(t)=-Kx(t).
\end{equation}
In this setup, \(x(t)\in \mathbb{R}^{n}\) represents the state vector, \(u(t)\in \mathbb{R}^{m}\) is the control input vector, \(w(t)\in \mathbb{R}^{l}\) denotes the exogenous disturbance, and \(z(t)\in \mathbb{R}^{q}\) is the controlled output. The matrices \(A \in \mathbb{R}^{n\times n}\), \(B_{1} \in \mathbb{R}^{n \times l}\), \(B_{2} \in \mathbb{R}^{n \times m}\), \(C  \in \mathbb{R}^{q \times n}\), and \(D  \in \mathbb{R}^{q \times m}\) are involved. Additionally, we consider a feedback gain matrix \(K\in \mathbb{R}^{m\times n}\) with a decentralized structure as defined:
\begin{equation}\label{a3}
	K =\left[
	\begin{array}{ccc}
		K_{1}&  & 0 \\
		& \ddots & \\
		0 &  &K_{m} \\
	\end{array}
	\right], 
\end{equation}
where \(K_i \in \mathbb{R}^{1 \times D_i}\), for \(i=1, \ldots, m\), and \(\sum^m_{i=1} D_i = n\).

\begin{assume}\label{ass:1}
The state matrix $C$ and control matrix $D$ satisfy $C^{\top}D=0,$ which means that there is no cross-weighting between the state variables and control variables, and the control weighting matrix $D$ satisfies $D^{\top}D\succ 0,$ with the standard assumptions that $(A,B_{2})$ is stabilizable and $(A,C)$ is detectable.
\end{assume}

\quad The objective function associated with the linear systems \eqref{a1} and \eqref{a2} is defined as
\begin{equation}\label{obj1}
J := \int z(t)^{\top}z(t) dt.
\end{equation}
By the Laplace transformation {\cite[Appendix C]{Desoer}}, minimizing the objective function \eqref{obj1} is equivalent to minimizing the following $\mathcal{H}_{2}$-guaranteed cost ODC problem
\begin{equation}\label{Pro5}
	\underset{K\in \mathscr{T}}{\mathrm{\min}}\|H(s)\|_{2}= \textrm{Tr}((C-DK)W_c(C-DK)^{\top}), 
\end{equation}
where the transfer function $H(s)=(C-DK)(sI-A+B_{2}K)^{-1}B_{1}$, and $W_c$ is the controllability Gramian matrix related to closed-loop system.
In order to ensure the robustness of the closed-loop system, we assume that the system matrices $(A, B_2)$ are unknown matrices with estimates $(\bar{A}, \bar{B}_2)$ available to the control designer. In particular, we consider 
\begin{equation*}
A = \bar{A}+\Delta_{A},B_2= \bar{B}_2+\Delta_{B_{2}},
\end{equation*}
in which the parametric uncertainty matrices $\Delta_{A}$ and $\Delta_{B_{2}}$ are unknown and belong to convex the following convex compact set of polytope type \cite{uncerLIn}, representable by 
\begin{align}
\mathcal{U}:=\bigg\{(\Delta_{A},\Delta_{B_{2}})\mid (\Delta_{A},\Delta_{B_{2}})= \sum_{i=1}^{M}\xi_{i}(\Delta_{A}^{i},\Delta_{B_{2}}^{i}), \sum_{i}^{M}\xi_{i}=1,~\xi_{i}\geq0\bigg\}\notag. 
\end{align}
Here $\left\{\Delta_{A}^{1},\ldots,\Delta_{A}^{M}\right\}$ and $\left\{\Delta_{B_{2}}^{1},\ldots,\Delta_{B_{2}}^{M}\right\}$ are known vertex matrices with fixed $M>0.$

% \subsection{Controller synthesis with convex parametrization}

\subsection{An approximated CP model of the ODC problem}
For notational convenience, we define the data matrix $\mathrm{\Phi}$ and parameter matrix $W$ by the following form:
\begin{align}  
	\mathrm{\Phi}:=\left[
	\begin{array}{cc}
		C^{\top}C &0  \\
		0 & D^{\top}D \\
	\end{array}
	\right]\in \mathbb{S}^{p}, \quad
	W:=\left[
	\begin{array}{cc}
		W_{1} & W_{2} \\
		W_{2}^{\top} & W_{3} \\
	\end{array}
	\right]\in \mathbb{S}^{p},\notag
\end{align}
where $W_{1}\in \mathbb{S}^{n}_{++}$, $W_{2}\in \mathbb{R}^{n\times m}$, $W_{3}\in \mathbb{S}^{m}$, $p=n+m.$ The following lemma shows that under an appropriate convex restriction, the feedback gain matrix $K$ derived from the parameter matrix $W$ can preserve the decentralized structure, robust stability, and robust performance, in the presence of parameter uncertainties.
\begin{lemma}{\cite[Theorem 1]{JMa}}\label{LemaJMa}
For $i=1,\ldots,M,$ the function $\mathcal{F}_{i}:\mathbb{S}^{p}\rightarrow \mathbb{R}^{n\times n}$ is defined as follows:
\begin{align}
	\mathcal{F}_{i}(W):=F_{i}WE^{\top}+EWF_{i}^{\top}+B_{1}B_{1}^{\top},\notag
\end{align}
where
\begin{align}
	E:=\left[\begin{array}{cc}
		I_{n\times n} & 0_{n\times m}\\
	\end{array}\right]\in \mathbb{R}^{n\times p}, \quad
	F_{i}:=\left[
	\begin{array}{cc}
		A^{i} & -B_{2}^{i} \\
	\end{array}
	\right]\in \mathbb{R}^{n\times p}. \notag
\end{align}
Based on the partition of the decentralized structure of the feedback gain matrix $K$ defined in \eqref{a3}, we can define the set
\begin{eqnarray} \label{Wset}
	\mathcal{W} &:=& \left\{ \; W \in \mathbb{S}^p \ | \  W \succeq 0, \; \mathcal{F}_{i}(W) \preceq 0, \; \forall \ i =1,2, \ldots, M,  \right.  \\
&& \quad W_{1}=\blkdiag \left\{ W_{1,D_1}, \, W_{1,D_2},\ldots, \, W_{1,D_m} \right\},\nonumber \\
&& \quad	W_{2} =\blkdiag \left\{W_{2,D_1}, \, W_{2,D_2}, \ldots, \, W_{2,D_m} \right\}, \nonumber  \\
&& \quad \left. W_{1,D_i}\in\mathbb{S}^{D_{i}}, \, W_{2,D_i}\in \mathbb{R}^{D_{i}}, \; \forall \ j=1,2, \ldots, m \right\} \nonumber 
\end{eqnarray}
and
\begin{equation*}
\mathcal{K} := \left\{W_{2}^{\top} W_{1}^{-1}  \ | \ W \in \mathcal{W} \right\}.
\end{equation*}
Then,
\begin{itemize}
\item[$(a)$] $K \in \mathcal{K}$ holds the decentralized structure in $(\ref{a3})$. 
\item[$(b)$] $K \in \mathcal{K}$ stabilizes the closed-loop system with uncertainties.
\item[$(c)$] $K \in \mathcal{K}$ gives $ \langle \mathrm{\Phi}, W  \rangle \ge
\|H_i(s)\|_{2}^{2}$, $\forall \ i=1,2, \ldots, M$, where $\| H_i(s) \|_2$ represents the $\mathcal{H}_2$-norm with repect to i-th extreme system.
\end{itemize}
\end{lemma}
\ni 

\quad On the one hand, we observe that $\langle \mathrm{\Phi}, W\rangle$ serves as an upper bound for $\|H(s)\|_{2}^{2}$ according to Lemma \ref{LemaJMa}. On the other hand, the block diagonal structure of $W_1$ and $W_2$ in \eqref{Wset} can be represented by the following linear systems:
\begin{align}\label{eqvij}
	V_{j1}WV_{j2}^{\top}=0, \quad j=1,\ldots,N,
\end{align}
where $N=(3m(m-1))/2,$ and the matrices $V_{j1}\in \mathbb{R}^{v_{j1}\times p}$ and $V_{j2}\in \mathbb{R}^{v_{j2}\times p}$ are provided in Appendix \ref{appxa}. Therefore, by introducing slack variables $S_{i}\in \mathbb{S}_{+}^{n}, i=1,\ldots, M,$ the design of the controller $K$ inherent in the $\mathcal{H}_{2}$-guaranteed cost decentralized control problem \eqref{Pro5} can be relaxed into the following SDP problem:
\begin{align}\label{pro22}
\underset{W,S_{i}}{\mathrm{\min}}~&\langle \mathrm{\Phi},W\rangle\notag\\
\textrm{s.t.}~
&S_{i}+\mathcal{F}_{i}(W)=0,\notag\\
&{V_{j1}WV_{j2}^{\top}}\hspace{0.39cm}=0,\hspace{0.24cm}j=1,\ldots,N,\\
&W\in \mathbb{S}_{+}^{p},~S_{i}\in \mathbb{S}_{+}^{n}, i=1,\ldots,M.\notag
\end{align}

\quad To reduce the heavy cost of matrix multiplication and to accelerate the practical performance, we consider rewriting the above problem in an equivalent vector form. For a symmetric matrix \(S \in \mathbb{S}^{n}\), there exists a matrix \(T_{n}\in\mathbb{R}^{\frac{n(n+1)}{2}\times n^{2}}\) such that 
\begin{align}\label{svec8}
&\textbf{svec}(S)= T_n\textbf{vec}(S) \quad \textrm{with} \quad T_{n}T_{n}^{\top}=I_{\frac{n(n+1)}{2}},
\end{align}
where the symmetric vectorization operator \(\textbf{svec} : \mathbb{S}^{n} \rightarrow \mathbb{R}^{\frac{n(n+1)}{2}}\) is defined by
\begin{align}
&\textbf{svec}(S):= \left[
           \begin{array}{ccccccc}
             s_{11},\sqrt{2}s_{21},\cdots,\sqrt{2}s_{n1},\cdots, s_{nn}
           \end{array}
         \right]^{\top}\notag.
\end{align}
For notational convenience, we define the following quantities:
\begin{align}
w&=\textbf{svec}(W)\in \mathbb{R}^{\frac{p(p+1)}{2}},~r=\textbf{svec}(\mathrm{\Phi})\in \mathbb{R}^{\frac{p(p+1)}{2}}, \hspace{0.4cm}b=\textbf{svec}(B_{1}B_{1}^{\top})\in \mathbb{R}^{\frac{n(n+1)}{2}},  \notag\\
	A_{w}^{i}&=T_{n}(E\otimes EF_{i}+EF_{i}\otimes E)T_{p}^{\top}\in \mathbb{R}^{{\frac{n(n+1)}{2}} \times\frac{p(p+1)}{2}},
	s_{i}=\textbf{svec}(S_{i})\in \mathbb{R}^{\frac{n(n+1)}{2}}, \notag\\
	B_{w}^{j}&={(V_{j2}\otimes V_{j1})T_{p}^{\top}}\in \mathbb{R}^{v_{j2}v_{j1}\times\frac{p(p+1)}{2}}, ~\textrm{for} ~i=1,\ldots,M, ~j=1,\ldots,N, \notag
\end{align}
and
\begin{align}
s &=\left[
	\begin{array}{ccccc}
		s_{1} &,& \ldots &,& s_{M} \\
	\end{array}
	\right]^{\top}\in\mathbb{R}^{M\frac{n(n+1)}{2}}, \; 	A_{w}=\left[
	\begin{array}{ccccc}
		A_{w}^{1} &,& \ldots &,& A_{w}^{M} \\
	\end{array}
	\right]^{\top}
	\in\mathbb{R}^{M\frac{n(n+1)}{2}\times\frac{p(p+1)}{2}},
	\notag\\
b_{w}& =\left[
\begin{array}{ccccc}
	b &, & \ldots &,&b \\
\end{array}
\right]^{\top}\in\mathbb{R}^{M\frac{n(n+1)}{2}}, \;
	\hspace{0.5cm}B_{w} =\left[
	\begin{array}{ccccc}
		B_{w}^{1} &,&\ldots&, & B_{w}^{N} \\
	\end{array}
	\right]^{\top}\in \mathbb{R}^{\sum_{j=1}^{N}v_{j2}v_{j1}\times\frac{p(p+1)}{2}}.\notag
\end{align}
Combining the properties of the Kronecker product with the equation \eqref{svec8}, we can express the problem \eqref{pro22} in the following vector form:
\begin{eqnarray}\label{ee6}
	\min_{w, s} && r^{\top} w \nonumber \\
	\textrm{s.t.}&	&A_{w}w+s+b_{w}=0, \\
&	&B_{w}w=0, \nonumber \\
&	&w\in \Gamma^{\frac{p(p+1)}{2}}, s\in \mathcal{C}, \nonumber
\end{eqnarray}
where the cone $\Gamma^{\frac{n(n+1)}{2}}\subseteq \mathbb{R}^{\frac{n(n+1)}{2}} $ is defined as
\begin{align}
\Gamma^{\frac{n(n+1)}{2}}&:=\left\{\textbf{svec}(X)\in \mathbb{R}^{\frac{n(n+1)}{2}}|\ X\in \mathbb{S}_{+}^{n}\right\}, \notag
\end{align}
and the convex set $\mathcal{C}\subseteq \mathbb{R}^{M\frac{n(n+1)}{2}}$ is denoted by $\mathcal{C}:=\Gamma^{\frac{n(n+1)}{2}}\times\ldots\times\Gamma^{\frac{n(n+1)}{2}}.$

\renewcommand{\arraystretch}{1}
\section{An accelerated PADMM for solving the CP model}\label{sec3}
In this section, we initially present the PADMM for solving the dual problem of the approximated CP model \eqref{ee6}. Subsequently, we establish the equivalence between the semi-proximal ADMM, including PADMM, and the (partial) PPA. This equivalence enables us to derive an accelerated PADMM by applying the Halpern iteration to the (partial) PPA.
\subsection{PADMM for solving the CP model}
The Lagrangian function of the problem (\ref{ee6}) defined by
\begin{equation*}
	\mathcal{L}(w_s;z,y,v_{\Lambda}):=-\langle z,b_{w}\rangle-\langle z+v,s\rangle+\langle r-\Lambda-A_{w}^{\top}z-B_{w}^{\top}y,w\rangle,
\end{equation*}
where the primal variable $w_s$, and multipliers $z,y$ and $v_{\Lambda}$ of the problem (\ref{ee6}) are defined as    
\begin{align}
z&:=[z_{1}; \ldots; z_M]\in \mathcal{Z}:=\mathbb{R}^{M\frac{n(n+1)}{2}}, \; y:=[y_1; \ldots; y_N]\in \mathcal{Y}:=\mathbb{R}^{\sum_{j=1}^{N}v_{j2}v_{j1}}, \notag\\
v_{\Lambda}&:= [\Lambda; v]\in \mathcal{V}:=\mathbb{R}^{\frac{p(p+1)}{2}}\times\mathcal{Z}, \hspace{0.54cm} w_{s}:= [w; s]\in \mathcal{V}.\notag
\end{align}
The dual of \eqref{ee6}, ignoring the minus sign in the objective, is given by
\begin{equation}
\label{DP0}
 \begin{aligned}
	\underset{z,y,\Lambda,v}{\mathrm{\min}}&\langle z,b_{w}\rangle\\
	~~\textrm{s.t.}~~
&A_{w}^{\top}z+B_{w}^{\top}y+\Lambda=r,\\
& \hspace{.5 cm} z  \hspace{1.15 cm} +  v\hspace{0.05cm}=0,\\
    & \hspace{.5 cm}\Lambda\in\Gamma^{\frac{p(p+1)}{2}}, v\in\mathcal{C}.
\end{aligned}   
\end{equation}
To obtain a more compact form of problem \eqref{DP0}, we introduce the linear operators $\mathcal{A}$ and $\mathcal{B}$ as follows: 
\begin{equation*}
	\mathcal{A}:= \left[\begin{array}{ccccc}
		A_{w}^{1} & I & 0& \cdots& 0 \\
		A_{w}^{2} & 0 &  I & \cdots & 0 \\
		\vdots & \vdots &\vdots  & \ddots  &\vdots\\
		A_{w}^{M} & 0& 0& \cdots& I
	\end{array} \right]= \left[\begin{array}{cc}
		A_{w} & \mathcal{I}
	\end{array}\right],  \quad 
\mathcal{B}:= \left[\begin{array}{cccc}
B_{w}^{1} & 0&  \cdots& 0 \\
\vdots & \vdots & \ddots  &\vdots\\
B_{w}^{N} & 0&  \cdots& 0
\end{array} \right] = \left[\begin{array}{cc}
B_{w} & 0
\end{array}\right].
\end{equation*}
The adjoints of the linear operators $\mathcal{A}$ and $\mathcal{B}$ are denoted as
\begin{align*}
\mathcal{A}^{*} = \left[\begin{array}{c}
	A_{w}^{\top} \\ \mathcal{I}
\end{array}\right]\quad\mathrm{and}\quad\mathcal{B}^{*} =\left[\begin{array}{c}
B_{w}^{\top} \\ 0
\end{array}\right].
\end{align*}
Let $\tilde{b}:= [ b_{w}; 0]\in \mathcal{Z}\times\mathcal{Y}$, $\tilde{r} := [r; 0] \in \mathcal{V},$ and
$\tilde{\mathcal{A}^{\ast}} := \left[
\begin{array}{cc}
\mathcal{A}^* & \mathcal{B}^* \\
\end{array}
\right].$ Then, we rewrite the problem \eqref{DP0} into the following two-block problem:
\begin{align}  \label{DP3}
\underset{\xi,v_{\Lambda}}{\mathrm{\min}} &~\langle \tilde{b}, \xi \rangle + \delta_{\Gamma}(v_{\Lambda}) \nonumber\\
\textrm{s.t.}&~~\tilde{\mathcal{A}}^* \xi + v_{\Lambda} = \tilde{r},
\end{align}
where $\Gamma=\Gamma^{\frac{p(p+1)}{2}}\times \mathcal{C}$ and $\xi = [z ; y]\in \mathcal{Z}\times\mathcal{Y}$. Given $\sigma>0$, for any $(v_{\Lambda},\xi; w_{s}) \in \mathcal{U}:=  \mathcal{V} \times (\mathcal{Z} \times \mathcal{Y}) \times \mathcal{V} $, we can define the following augmented Lagrangian function associated with the problem \eqref{DP3}:
\begin{align}
\mathcal{L}_\sigma (v_{\Lambda},\xi; w_{s}) &:=  \langle \tilde{b}, \xi \rangle + \delta_{\Gamma}(v_{\Lambda}) + \langle w_{s},  \tilde{\mathcal{A}}^* \xi + v_{\Lambda}- \tilde{r} \rangle+\frac{\sigma}{2} \| \tilde{\mathcal{A}}^* \xi + v_{\Lambda}- \tilde{r}\|^2. \notag
\end{align}
Then, the KKT system associated with \eqref{DP3} can be defined as follows:
\begin{equation*}
\left\{
\begin{aligned}
0&= \tilde{b} + \tilde{\mathcal{A}} {w}_{s},\\
0&\in \partial \delta_{\Gamma}({v}_{\Lambda})+ {w}_{s},\\
0&= \tilde{r} - \tilde{\mathcal{A}}^* \xi - {v}_{\Lambda}.
\end{aligned}
\right.
\end{equation*}
Let $\Omega$ be the solution set to the above KKT system. Suppose that  $\Omega\neq\emptyset$. We define the following monotone operator $\mathcal{T}$ by
\begin{equation*}
\mathcal{T}(v_{\Lambda},\xi; w_{s}):=
\left(
  \begin{array}{c}
    \tilde{b} + \tilde{\mathcal{A}}{w}_{s}\\
    \partial \delta_{\Gamma}({v}_{\Lambda})+ {w}_{s}\\
    \tilde{r} - \tilde{\mathcal{A}}^* \xi - {v}_{\Lambda}\\
  \end{array}
\right), \quad \forall (v_{\Lambda},\xi; w_{s}) \in \mathcal{U},
\end{equation*}
where $\partial \delta_{\Gamma}(\cdot)$ denotes the subdifferential of convex function $\delta_{\Gamma}(\cdot)$. It is clear that $\mathcal{T}$ is a maximally monotone operator \cite[Example 20.26 and Corollary 25.5]{BC17} and $\mathcal{T}^{-1}(0) = \Omega \neq \emptyset.$ 

\quad Now, we focus on the design of the algorithm. To enhance computational efficiency, we will select distinct proximal terms based on the problem scale, categorizing them into medium-scale problems ({\bf Case 1}) and large-scale problems ({\bf Case 2}).
\begin{enumerate}
  \item []{\bf Case 1.} Let $\mu_0$ and $\mu_1$ be given positive parameters.  Define
\begin{equation*}
\tilde{\mathcal{S}}_0 := \sigma{\mu_{0} \left[ \begin{array}{cc }
\mathcal{I}_{\mathcal{Z}}  & 0  \\
0  & \mathcal{I}_{\mathcal{Y}}
\end{array}\right]},
\end{equation*}
where $\mathcal{I}_{\mathcal{Z}}$, $\mathcal{I_{\mathcal{Y}}}$ and $\mathcal{I_{\mathcal{V}}}$ denote the identity operators in $\mathcal{Z}$, $\mathcal{Y}$, and $\mathcal{V}$, respectively. Then, a two-block PADMM for solving medium-scale problem \eqref{DP3} is displayed in Algorithm \ref{alg1}:
\begin{algorithm}
\caption{PADMM\_2blk}
\label{alg1}
\begin{algorithmic}[h]
\STATE Set $u^{0} = (v_{\Lambda}^{0}, \xi^{0}, w_{s}^{0}) \in \mathcal{U}$ to be the initial point.\\
$k=0,1,...,$
\begin{align}
				{v}_{\Lambda}^{k+1}&=\underset{ {v}_{\Lambda}\in \mathcal{V} }{\mathrm{argmin}}~\mathcal{L}_\sigma (v_{\Lambda},{\xi}^{k}; {w}_{s}^{k})  + \frac{1}{2} \|{v}_{\Lambda} - v^k_{\Lambda} \|^2_{\mu_1 \mathcal{I}_\mathcal{V}}\notag\\
				{w}_{s}^{k+1}&= w_{s}^k + \sigma (\tilde{\mathcal{A}}^*{\xi}^{k} + {v}_{\Lambda}^{k+1}  -\tilde{r} )\notag\\
				\displaystyle{\xi}^{k+1}&=  \underset{ \xi\in \mathcal{Z}\times\mathcal{Y} }{\mathrm{argmin}}~\mathcal{L}_\sigma ({v}_{\Lambda}^{k+1},\xi;{w}_{s}^{k+1}) + \frac{1}{2} \|{\xi}- \xi^{k} \|^2_{{\tilde{\mathcal{S}}}_0}\notag
\end{align}
\end{algorithmic}
\end{algorithm}
  \item []{\bf Case 2.} Let $\mu_2$ and $\mu_3$ be given positive parameters. Define 
 \begin{eqnarray*}	
\mathcal{Q} & := & \sigma \tilde{\mathcal{A}} \tilde{\mathcal{A}}^* + \tilde{\mathcal{S}}'_0  =
 \sigma \left[\begin{array}{cc}\mathcal{A} \mathcal{A}^* + \mu_2 \mathcal{I}_{\mathcal{Z}} &  \mathcal{A}\mathcal{B}^* \\ \mathcal{B}\mathcal{A}^*  & \mathcal{B}\mathcal{B}^* + \mu_3 \mathcal{I}_{\mathcal{Y}} \end{array} \right]= \mathcal{Q}_u + \mathcal{Q}_d + \mathcal{Q}^*_u, 
	\end{eqnarray*}
where
\begin{equation*}
\tilde{\mathcal{S}}'_0  :=\sigma\left[ \begin{array}{cc }
\mu_2	\mathcal{I}_{\mathcal{Z}}  & 0  \\
0  & \mu_3 \mathcal{I}_{\mathcal{Y}}
\end{array}\right],\quad \mathcal{Q}_u :=  \sigma \left[\begin{array}{cc}0 &\mathcal{A}\mathcal{B}^* \\0 & 0\end{array}\right], \mathrm{and}~ \mathcal{Q}_d:=  \sigma \left[\begin{array}{cc} \mathcal{A} \mathcal{A}^* + \mu_2  \mathcal{I}_{\mathcal{Z}} & 0 \\0 & \mathcal{B}\mathcal{B}^* + \mu_3 \mathcal{I}_{\mathcal{Y}}\end{array}\right].
\end{equation*} 
Note that $\mathcal{Q}_d$ is positive definite. From \cite{Lischur,Lisgs}, the self-adjoint positive semi-definite linear operator $\textrm{sGS}(\mathcal{Q}): \mathcal{Z}\times\mathcal{Y} \rightarrow \mathcal{Z}\times\mathcal{Y}$ can be defined as
\begin{align}\label{sGSop}
	\textrm{sGS}(\mathcal{Q}):= \mathcal{Q}_u \mathcal{Q}_d^{-1}\mathcal{Q}_u^*.
\end{align}	
For large-scale problems, addressing the linear system involving $\xi$ poses a significant challenge. Utilizing the $\textrm{sGS}(\mathcal{Q})$ operator enables the separation of $\xi$ into $z$ and $y$. Consequently, problem \eqref{DP3} can be solved by the following multi-block PADMM provided in Algorithm \ref{alg2}:
\begin{algorithm}[H]
\caption{PADMM\_sGS}
\label{alg2}
\begin{algorithmic}
\STATE Set $u^{0} = (v_{\Lambda}^{0}, (z^{0}, y^{0}),  w_{s}^{0})\in \mathcal{U}$ to be the initial point.\\
$k=0,1,...,$
			\begin{eqnarray}
		{v}_{\Lambda}^{k+1} &=& \underset{{v}_{\Lambda}\in \mathcal{V} }{\mathrm{argmin}}~\mathcal{L}_\sigma (v_{\Lambda}, (z^k, y^k); {w}_{s}^{k})  + \frac{1}{2} \|{v}_{\Lambda} - v^k_{\Lambda} \|^2_{\mu_1 \mathcal{I}_\mathcal{V}}\notag\\
 	{w}_{s}^{k+1}&=& w_{s}^k + \sigma (\tilde{\mathcal{A}}^*[{z}^{k};y^k] + {v}_{\Lambda}^{k+1}  -\tilde{r} )\notag\\
	\displaystyle {y}^{k+\frac{1}{2}}&=& \underset{y\in \mathcal{Y} }{\mathrm{argmin}}~\mathcal{L}_\sigma ( {v}_{\Lambda}^{k+1},(z^k,y);{w}_{s}^{k+1}) + \frac{1}{2} \|y- y^{k} \|^2_{\sigma\mu_3 \mathcal{I}_\mathcal{Y}}\label{sGSPADMMa}\\	
 {z}^{k+1}&=& \underset{z\in \mathcal{Z} }{\mathrm{argmin}} ~\mathcal{L}_\sigma  ({v}_{\Lambda}^{k+1}, (z, y^{k+\frac{1}{2}});{w}_{s}^{k+1}) + \frac{1}{2} \|{z}- z^{k} \|^2_{\sigma\mu_2 \mathcal{I}_\mathcal{Z}}\label{sGSPADMMb}\\	
		\displaystyle{y}^{k+1}&=& \underset{y\in \mathcal{Y} }{\mathrm{argmin}}~ \mathcal{L}_\sigma ({v}_{\Lambda}^{k+1}, (z^{k+1}, y);{w}_{s}^{k+1}) + \frac{1}{2} \|{y}- y^{k} \|^2_{\sigma\mu_3 \mathcal{I}_\mathcal{Y}}\label{sGSPADMMc}
\end{eqnarray}
\end{algorithmic}
\end{algorithm}
\end{enumerate} To simplify the notation in the subsequent discussion, we introduce the following proximal operators
\begin{equation}\label{opS}
	\mathcal{S}_1 = {\mu_{1} \mathcal{I}_{\mathcal{V}}},\quad \text{ and }\quad \mathcal{S}_2:=
	\left\{
	\begin{aligned}
		&\tilde{\mathcal{S}}_0,                              \hspace{1.96cm} \textrm{for PADMM\_2blk,}\\
		&\textrm{sGS}({\mathcal{Q}})+ \tilde{\mathcal{S}}'_0, \quad \textrm{for PADMM\_sGS}.
	\end{aligned}
	\right.
\end{equation} 
This allows us to consolidate Algorithm \ref{alg1} and Algorithm \ref{alg2} into the unified PADMM framework presented in Algorithm \ref{alg3}:
\begin{algorithm}[h]
\caption{PADMM}
\label{alg3}
\begin{algorithmic}[]
		\STATE Set $u^{0}=(v_{\Lambda}^{0},\xi^{0},w_{s}^{0})\in \mathcal{U}$ to be the initial  point.\\
$k=0,1,...,$
		\begin{align}
			{v}_{\Lambda}^{k+1}&=\underset{ {v}_{\Lambda}\in \mathcal{V} }{\mathrm{argmin}}~\mathcal{L}_\sigma (v_{\Lambda},{\xi}^{k}; {w}_{s}^{k})  + \frac{1}{2} \|{v}_{\Lambda} - v^k_{\Lambda} \|^2_{\mathcal{S}_1}\label{PADMMa}\\
			{w}_{s}^{k+1}&= w_{s}^k + \sigma (\tilde{\mathcal{A}}^*{\xi}^{k} + {v}_{\Lambda}^{k+1}  -\tilde{r} )\notag\\
			\displaystyle{\xi}^{k+1}&=\underset{ \xi\in \mathcal{Z}\times\mathcal{Y} }{\mathrm{argmin}}~\mathcal{L}_\sigma ({v}_{\Lambda}^{k+1},\xi;{w}_{s}^{k+1}) + \frac{1}{2} \|{\xi}- \xi^{k} \|^2_{\mathcal{S}_2}\label{PADMMc}
		\end{align}
\end{algorithmic}
\end{algorithm}\\

\begin{remark}
The selection of ${\mathcal{S}}_2$ depends primarily on the problem dimensions and computer RAM.  For example, in our numerical experiments, the dual problem \eqref{DP3} and the primal problem \eqref{ee6} generated from $n=300, m=2,  M=4$ in the ODC problem \eqref{Pro5} can be treated as a large-scale problem, for which the proximal operator $\textrm{sGS}({\mathcal{Q}})+ \tilde{\mathcal{S}}'_0$ is used in the PADMM. Then, the large-scale subproblems involving this proximal operator can be efficiently solved by leveraging the sGS decomposition technique developed in \cite{Lischur, Lisgs}.
\end{remark}

\subsection{Acceleration of the PADMM scheme}
Define a self-adjoint linear operator $\mathcal{S}: \mathcal{U}\to \mathcal{U}$ in the following form:
\[
\mathcal{S} := \left(
\begin{array}{ccc}
	\mathcal{S}_1  &  & \\
	&\sigma \tilde{\mathcal{A}}\tilde{\mathcal{A}}^* + \mathcal{S}_2 & \tilde{\mathcal{A}} \\
	& \tilde{\mathcal{A}}^* &  \sigma^{-1} \mathcal{I}_{\mathcal{V}} \\
\end{array}
\right).
\]
Since $\mathcal{S}_1$ and $\mathcal{S}_2$ are two self-adjoint positive definite linear operators, according to the equivalence between the semi-proximal ADMM and the (partial) proximal point method presented in Proposition \ref{prop-B} in Appendix \ref{appxb}, we can obtain the following equivalence:
\begin{prop}\label{prop-B1}
 Consider $(v_\Lambda^0,\xi^0,w_s^0) \in   \mathcal{U}$ and $\rho\in \mathbb{R}$. Then the point $(v_\Lambda^+, \xi^+, w_s^+)$ generated by the following generalized PADMM (GPADMM) scheme
	\begin{equation}
		\label{padmm}
		\left\{
		\begin{aligned}
			\bar{v}_\Lambda^{+} ={}&\underset{{v}_{\Lambda}\in \mathcal{V} }{\mathrm{argmin}}~\mathcal{L}_\sigma (v_\Lambda, \xi^0; w_s^0) + \frac{1}{2}\norm{v_\Lambda - v_\Lambda^0}_{\mathcal{S}_1}^2, \\[5pt]
			\bar{w}_s^{+} ={}& w_s^0 + \sigma( \bar{v}_\Lambda^{+} + \tilde{\mathcal{A}}^* \xi^0 - \tilde{r}),\\[5pt]
			\bar{\xi}^{+} ={}&\underset{ \xi\in \mathcal{Z}\times\mathcal{Y} }{\mathrm{argmin}}~\mathcal{L}_{\sigma} (\bar{v}_\Lambda^{+}, \xi; \bar{w}_s^{+} ) + \frac{1}{2}\norm{\xi-  \xi^0}_{\mathcal{S}_2}^2, \\[5pt]
			&\hspace{-1.0cm}(v_\Lambda^+, \xi^+, w_s^+) = (1-\rho)(v_\Lambda^0, \xi^0, w_s^0)
			+ \rho (\bar{v}_\Lambda^+, \bar{\xi}^+, \bar{w}_s^+)
		\end{aligned}
		\right.
	\end{equation}
	is equivalent to the one generated by the following proximal point scheme
	\begin{equation}\label{ppa-r}
		\left\{\begin{aligned}
			&0\in \mathcal{T} (\bar{v}_\Lambda^+, \bar{\xi}^+, \bar{w}_s^+) + \mathcal{S}(\bar{v}_\Lambda^+ - {v}_\Lambda^0, \bar{\xi}^{+} - {\xi}^0,\bar{w}_s^+ - {w}_s^0),
			\\[5pt]
			&(v_\Lambda^+, \xi^+, w_s^+) = (1-\rho)(v_\Lambda^0, \xi^0, w_s^0) + \rho (\bar{v}_\Lambda^+, \bar{\xi}^+, \bar{w}_s^+).
		\end{aligned}
		\right.
	\end{equation}
\end{prop}

\quad To achieve a faster convergence rate, based on the above Proposition \ref{prop-B1}, we shall employ the Halpern fixed-point method to accelerate the proximal point scheme \eqref{ppa-r}. Let $\rho$ be a given relaxation parameter in $(0,2]$ and $u := (v_\Lambda, \xi, w_s)\in \mathcal{U}$. Then
each iteration in the relaxed PPA \eqref{ppa-r} can be written in an abstract form: 
\[
u^+ = F_\rho (u),
\]
where the operator $F_\rho:\cal U \to \cal U$ is defined by
\begin{align}\label{relop}
	{F_\rho}({u})&:= (1-\rho)u +\rho  {(\mathcal{S}+ \mathcal{T})^{-1}} \,  \mathcal{S}u  =u- 2 \cdot \frac{\rho}{2} \cdot[{\textcolor{red}{}}\mathcal{I} - (\mathcal{S}+ \mathcal{T})^{-1}\mathcal{S}] {u}.
\end{align}
Note that
\[
u\in {\rm Fix}(F_\rho) \; \Leftrightarrow \; 
u = ({\cal S} + {\cal T})^{-1}{\cal S}u \;  \Leftrightarrow  \;  0\in {\cal T}(u).
\]
Hence \[
{{\textrm{Fix}(F_\rho) = {\cal T}^{-1}(0)}}.
\]
For notational convenience, we drop the dependence of $F$ on $\rho$ in the rest of this section.

\quad Since $\cal S$ is a symmetric positive definite linear operator, we know that ${\cal S}^{-1}$ can be decomposed as ${\cal S}^{-1} = LL^T$ with $L$ being a real lower triangular matrix with positive diagonal entries.
Next, we shall show that $F$ is in fact non-expansive with respect to the induced norm $\|\cdot\|_{\cal S}$.
For this purpose, we define the normalized operator of $F$ by $\widetilde F:\cal U\to \cal U$ in the following way:
\begin{align}\label{equa52}
	\widetilde{F}(\tilde{u})&:=  \displaystyle (1-\rho)\tilde{u} +\rho \left( \mathcal{I}  +  L^T \mathcal{T} L\right)^{-1} \tilde{u} =\tilde{u}- 2 \cdot \frac{\rho}{2} \cdot \left[ \mathcal{I}-( \mathcal{I}+\widetilde{\mathcal{T}})^{-1} \right]\tilde{u},
\end{align}
where $\widetilde{\mathcal{T}} = L^T \mathcal{T} L$ is a maximally monotone operator due to \cite[Proposition 23.25]{BC17}.
Similar to \cite[Lemma 2.1]{Lisnipal},
it is not difficult to verify that
\begin{equation}\label{eq:tildeF}
L^{-1}F(u)=\widetilde{F}(L^{-1}u), \quad \forall \, u\in {\cal U}.
\end{equation}

\begin{prop}\label{prop:non-ex}
Given the parameter $\rho\in(0,2]$, the operator $\widetilde{F}$ defined in \eqref{equa52} is non-expansive. Moreover, $F$ is non-expansive with respect to the induced norm $\|\cdot\|_{\cal S}$, i.e.,
\begin{equation*}
	\| F(u) - F(v)\|_{\cal S} \le \|u -v\|_{\cal S}, \quad \forall\, u, v\in {\cal U}.
\end{equation*}
\end{prop}

\begin{proof}
Define the resolvent operator of $ \widetilde{\mathcal{T}}$ as $\widetilde{G} :=( \mathcal{I}+\widetilde{\mathcal{T}})^{-1}$ and $\widetilde{\mathcal{Q}}:= \mathcal{I}-\widetilde{G}.$ It should be noted that $\widetilde{G}$ is a single-valued and firmly non-expansive operator.
By \cite[Proposition 4.4]{BC17} and $\rho\in(0,2],$ we have
\begin{align}
\langle \widetilde{\mathcal{Q}}x- \widetilde{\mathcal{Q}}y , x- y \rangle&=\langle \widetilde{\mathcal{Q}} x - \widetilde{\mathcal{Q}}y  ,  \widetilde{G}x - \widetilde{G} y\rangle +\|  \widetilde{\mathcal{Q}}x - \widetilde{\mathcal{Q}}y \|^2\notag\\
&\geq\|\widetilde{\mathcal{Q}} x -\widetilde{\mathcal{Q}}y \|^2\geq\frac{\rho}{2}\|\widetilde{\mathcal{Q}} x-\widetilde{\mathcal{Q}}y \|^2\notag,
\end{align}
which  shows that $\widetilde{\mathcal{Q}}$ is $\frac{\rho}{2} $-co-coercive. Then,  according to \cite[Proposition 4.11]{BC17}, we know that the mapping $\widetilde{F}:=  \mathcal{I}-2(\frac{\rho}{2})\widetilde{\mathcal{Q}}$ is a non-expansive operator.

\quad Now for all $u,v\in{\cal U}$, it holds from \eqref{eq:tildeF} that
\begin{align*}
\|F(u) - F(v)\|_{\cal S} = {}&\|L^{-1} F(u) - L^{-1} F(v)\| \\
= {}&\| \widetilde{F}(L^{-1}u) - \widetilde{F}(L^{-1}v)\| \\
\le {}&\|L^{-1} u - L^{-1} v\| = \|u - v\|_{\cal S},
\end{align*}
where the last inequality follows from the non-expansiveness of $\widetilde F$.
\end{proof}

\quad Proposition \ref{prop:non-ex} implies that the following Halpern fixed-point iteration can be used to find a fixed-point in ${\rm Fix}(F)$:
\begin{align}\label{hal}
{u}^{k+1}=\frac{1}{k+2}{u}^{0} +\left(1-\frac{1}{k+2}\right)F(u^k), \quad \forall \, k\ge 0.
\end{align}
Then, the definition of $F$ in \eqref{relop}, together with Proposition \ref{prop-B1}, implies that the update scheme \eqref{hal} can be equivalently recast as the following accelerated PADMM algorithm.

\begin{algorithm}[H]
\caption{the accelerated proximal ADMM (APADMM)}\label{alg4}
\begin{algorithmic}
\STATE Choose parameters $\sigma>0,$  $\mu_0>0$, $\mu_1>0$, $\mu_{2}>0,$ $\mu_3>0$, and $\rho\in(0,2]$. Let operators $\mathcal{S}_1$ and $ \mathcal{S}_2$ be defined in \eqref{opS}. Select an initial point $u^{0}=(v_{\Lambda}^{0},\xi^{0},w_{s}^{0})\in\mathcal{U}.$\\
$k=0,1,...,$
\begin{align}
&\bar{v}_{\Lambda}^{k}=\underset{{v}_{\Lambda}\in \mathcal{V} }{\mathrm{argmin}}~\mathcal{L}_\sigma (v_{\Lambda},{\xi}^{k}; {w}_{s}^{k})  + \frac{1}{2} \|{v}_{\Lambda} - v^k_{\Lambda} \|^2_{\mathcal{S}_1}\label{ProPADMM}\\
&\bar{w}_{s}^{k}= w_{s}^k + \sigma (\tilde{\mathcal{A}}^*{\xi}^{k} + \bar{v}_{\Lambda}^{k}-\tilde{r} )\notag\\
&\displaystyle\bar{\xi}^{k}=\underset{\xi\in \mathcal{Z}\times\mathcal{Y} }{\mathrm{argmin}}~\mathcal{L}_\sigma (\bar{v}_{\Lambda}^{k},\xi;\bar{w}_{s}^{k}) + \frac{1}{2} \|{\xi}  - \xi^{k} \|_{\mathcal{S}_2}^{2}\label{PADMMb}\\
&\hat{u}^{k+1}=\rho\bar{u}^{k}+(1-\rho){u}^{k}\notag\\
&{u}^{k+1}=\frac{1}{k+2}{u}^{0}+\frac{k+1}{k+2}\hat{u}^{k+1}\notag
\end{align}
\end{algorithmic}
\end{algorithm}

\begin{remark}
Algorithm \ref{alg4} incorporates both the relaxation and the acceleration techniques. We shall remark that the relaxation parameter $\rho$ can be chosen
in $(0,2]$. This is the key feature that distinguishes it from the classic generalized ADMM \cite{DRA92} where $\rho$ is restricted in $(0,2)$. Furthermore, to improve the performance of {Algorithm} \ref{alg4}, in our implementation, the restarting strategy studied in \cite[Section 11.4]{Nemirov} and \cite[Section 5.1]{Neste} is also employed.
\end{remark}

Now we are ready to state the convergence and rate of convergence for  $\textrm{Algorithm}$ \ref{alg4}.
\begin{thm}{\cite[Theorem 2]{Wittmann}} \label{lem:wit}
Assume that ${\cal T}^{-1}(0)\neq \emptyset$. Let $\{u^k\}_{k=0}^{\infty}$ be the infinite sequence generated by Algorithm \ref{alg4}. Then, it holds that
$\lim_{k\to \infty}u^k = \Pi_{{\cal T}^{-1}(0)}(u^0)$.
\end{thm}

By {\cite[Theorem 2.1]{Lied21}}, we have the following results on the rate of convergence for Algorithm \ref{alg4}.

\begin{thm}\label{Coll}
Assume that ${\cal T}^{-1}(0)\neq \emptyset$. Let $\{u^k\}_{k=0}^{\infty}$ be the infinite sequence generated by Algorithm \ref{alg4}. Then, it holds that
\begin{align*}
\|u^{k}-F(u^{k})\|_{\mathcal{S}}\leq\frac{2\|u^{0}-u^{*}\|_{\mathcal{S}}}{k+1},~\forall\, k\geq0 \; \mathrm{and} \; u^{*}\in {\cal T}^{-1}(0).
\end{align*}
\end{thm} 
\section{A fast implementation for solving the subproblems of APADMM}\label{sec4}

In this section, we provide a fast implementation for solving subproblems of APADMM for solving problem \eqref{DP3} with different proximal operators $\mathcal{S}_2$ depending on the problem scale.  

\subsubsection*{The projection onto the cone $\Gamma$}

In the subproblem \eqref{ProPADMM}, the variable $\bar{\Lambda}^{k}$ is obtained through the following projection:
\begin{align}\label{ulambda}
\bar{\Lambda}^{k}=\Pi_{\Gamma^{\frac{p(p+1)}{2}}}\left[\frac{\mu_{1}\Lambda^{k}-\sigma(A_{w}^{\top}{z}^{k}+B_{w}^{\top}{y}^{k}-r+\frac{1}{\sigma}{w}^{k})}{\sigma+\mu_{1}}\right],
\end{align}
and the variable $\bar{v}^{k}=[\bar{v}_{1}^{k};\ldots;\bar{v}_{M}^{k}]$ is computed as follows: for $i\ge 1$,
\begin{equation}\label{upv}
\bar{v}_{i}^{k}=\Pi_{\Gamma^{\frac{n(n+1)}{2}}}\left[\frac{\mu_{1} v_{i}^{k}-(\sigma{z}_{i}^{k}+{s}_{i}^{k})}{\sigma+\mu_{1}}\right].
\end{equation}

\subsubsection*{A lifting technique for solving large-scale sparse linear systems}

To enhance the efficiency of the algorithm, we introduce a lifting technique to solve large-scale sparse linear systems in addressing the subproblem \eqref{PADMMb}, which is suitable for both {\bf Case 1} and {\bf Case 2} as mentioned before.

\quad {\bf Case 1.}\
In the APADMM iteration scheme with the operator $\mathcal{S}_{2}=\tilde{\mathcal{S}}_0$, the optimality condition for $\bar{\xi}^{k}$ in \eqref{PADMMb} is
\begin{align}\label{bigs}
\left(\tilde{\mathcal{A}}\tilde{\mathcal{A}}^*+\frac{1}{\sigma}\mathcal{\tilde{S}}_{0}\right){\bar{\xi}}^{k}
=\frac{1}{\sigma}\mathcal{\tilde{S}}_{0}\xi^{k}-\frac{1}{\sigma}\left(\tilde{b}
+\tilde{\mathcal{A}}w_{s}^{k}\right)-\tilde{\mathcal{A}}(v_{\Lambda}^{k}-\tilde{r}).
\end{align}
Note that calculating $\tilde{\mathcal{A}}\tilde{\mathcal{A}}^*$ is time-consuming, and even if $\tilde{\mathcal{A}}$ is sparse, $\tilde{\mathcal{A}}\tilde{\mathcal{A}}^*$ turns out to be dense for problem \eqref{DP3}. Consequently, solving the linear system \eqref{bigs} by using direct methods becomes challenging. To deal with this challenge, we introduce the auxiliary variable $\eta^{\prime}=\tilde{\mathcal{A}}^{*}{\bar{\xi}}^{k}$
and solving the following augmented system
\begin{align}\label{bigslin}
&\left[
\begin{array}{cc}
\frac{1}{\sigma}\mathcal{\tilde{S}}_{0}& \tilde{\mathcal{A}} \\
\tilde{\mathcal{A}}^{\ast} & -\mathcal{I} \\
\end{array}
\right]\left[
\begin{array}{c}
{\bar{\xi}}^{k} \\
\eta^{\prime} \\
\end{array}
\right]=\left[
\begin{array}{c}
\frac{1}{\sigma}\mathcal{\tilde{S}}_{0}\xi^{k}-\frac{1}{\sigma}\left(\tilde{b}+\tilde{\mathcal{A}}w_{s}^{k}\right)-\tilde{\mathcal{A}}(v_{\Lambda}^{k}-\tilde{r}) \\
0 \\
\end{array}
\right].
\end{align}
This lifting technique allows us to avoid the computations of $\tilde{\mathcal{A}}\tilde{\mathcal{A}}^*$ and the matrix-vector product $\tilde{\mathcal{A}}^{*}{\bar{\xi}}^{k}$ in the update of $\bar{\Lambda}^{k}$.

\quad {\bf Case 2.}\ For large-scale problems, when substituting the proximal operator ${\cal S}_{2} = \textrm{sGS}({\mathcal{Q}})+ \tilde{\mathcal{S}}'_0$ into the subproblem \eqref{PADMMb}, we observe that $\bar{\xi}^{k}=[\bar{z}^{k}; \bar{y}^{k}]$ in Algorithm \ref{alg4} is equivalent to ${\xi}^{k+1}=[{z}^{k+1}; {y}^{k+1}]$ in Algorithm \ref{alg2}. Therefore, we can obtain $\bar{\xi}^{k}$ by solving \eqref{sGSPADMMa}-\eqref{sGSPADMMc} in Algorithm \ref{alg2} according to sGS decomposition approach. In particular, from the equation \eqref{sGSPADMMa}, $y^{k+1/2}$  can be obtained via  solving the following problem
\begin{align}
y^{k+1/2}&=\underset{y\in \mathcal{Y}}{\mathrm{argmin}}\left\{\frac{\sigma}{2}\|B_{w}^{\top}y+H_1^{b}\|^{2}+\frac{1}{2}\|y-y^{k}\|^2_{{{\sigma\mu_{3} {{\mathcal{I}_{\mathcal{Y}}}}}}}\right\}\notag.
\end{align}
The optimality condition of this problem implies that $y^{k+1/2}$ is the solution to the following linear equation
\begin{align}
{{(\mu_{3}\mathcal{I_{Y}}+B_{w}B_{w}^{\top})y=\mu_{3} y^{k}-B_{w}H_1^{b}}}\notag,
\end{align}
where $H_1^{b} = -r+\frac{1}{\sigma}{w}^{k+1}+{\Lambda}^{k+1} +A_{w}^{\top}z^{k}.$ Additionally, the update of $z^{k+1}$ in \eqref{sGSPADMMb} can be written explicitly in the following form:
\begin{align}
{z}^{k+1}=\underset{z\in \mathcal{Z}}{\mathrm{argmin}}&\left\{\langle z,b_{w}\rangle+\frac{\sigma}{2}\|z+\frac{1}{\sigma}{s}^{k+1}+{v}^{k+1}\|^{2}+\frac{\sigma}{2}\|A_{w}^{\top}z+H_2^{b}\|^{2}+\frac{1}{2} \|z-z^{k} \|^2_{\sigma\mu_{2}{\mathcal{I}_{\mathcal{Z}}}}\right\}\notag.
\end{align}
{From its optimality condition}, we know that
\begin{align*}
&\left[(\mu_{2}+1)\mathcal{I_{Z}}+A_{w}A_{w}^{\top}\right]{z}^{k+1}=\mu_{2}z^{k}-\left[{v}^{k+1}+\frac{1}{\sigma}(b_{w}+{s}^{k+1})+A_{w}H_2^{b}\right].
\end{align*}
To update ${z}^{k+1}$, we can employ a similar lifting technique as used in \eqref{bigs} to solve the following linear system:
\begin{align}\label{equ32}
&\left[
  \begin{array}{cc}
    (\mu_{2}+1)\mathcal{I}_{\mathcal{Z}}& A_{w} \\
    A_{w}^{\top} & -\mathcal{I} \\
  \end{array}
\right]\left[
         \begin{array}{c}
           {z}^{k+1} \\
           z^{\prime} \\
         \end{array}
       \right]=\left[
                 \begin{array}{c}
                   \mu_{2}z^{k}-(v^{k}+\frac{1}{\sigma}(b_{w}+s^{k})+A_{w}H_{2}^{b}) \\
                   0 \\
                 \end{array}
               \right],
\end{align}
where $H_2^{b} = -r+\frac{1}{\sigma}{w}^{k+1}+{\Lambda}^{k+1} + B_{w}^{\top}y^{k+1/2}.$
Finally, we obtain $y^{k+1}$ similarly to $y^{k+1/2}$ as follows:
\begin{align}
{{(\mu_{3} \mathcal{I_{Y}}+ B_{w}B_{w}^{\top})y^{k+1}=\mu_{3} y^{k}- B_{w}H_{1}^{f}}}\notag,
\end{align}
where $H_{1}^{f} = -r+\frac{1}{\sigma}{w}^{k+1}+{\Lambda}^{k+1} +A_{w}^{\top}{z}^{k+1}.$
\vspace{4pt}
% \begin{remark}
% There are significant differences between the linear systems appeared in  PADMM\_2blk and PADMM\_sGS. In PADMM\_2blk, the variable $\xi$ consists of both $y$ and $z$, which results in a larger linear system \eqref{bigslin}. In contrast, PADMM\_sGS utilizes the sGS decomposition technique \cite{Lischur,Lisgs} to break down a large linear system into smaller ones corresponding to $y$ and $z$ separately, and solve them individually. Numerical experiments suggest that PADMM\_2blk is suitable for solving small and medium-sized problems, while PADMM\_sGS is better suited for addressing large-scale problems.	
% \end{remark}

\section{Numerical experiments}\label{sec5}

In this section, we present the numerical performance of APADMM in solving ODC problems for linear time-invariant systems. The APADMM implementation was done in C language (using the MSVC compiler) and executed on a workstation equipped with an Intel(R) Core(TM) i9-10900 CPU@2.80GHz and 64GB RAM. 

\subsection{Implementation details}
For the data in ODC problems \eqref{a1}, matrices $A$ and $B_2$ are generated following a normal distribution with zero mean and unit variance. Subsequently, $C$ and $D$ are obtained by a random orthogonal matrix to satisfy $C^{\top}D=0$. Additionally, we generate $B_1$ to satisfy $0\succeq\mathcal{F}_{i}(W)$ with $i=1,\dots,M$ for a random $W\in \mathbb{S}^{p}_+$.

\quad After generating the data, we compared the performance of APADMM with GPADMM in \eqref{padmm}, sGS\_PADMM \cite{Lischur}, and popular solvers such as SCS (C source code version 3.2.3), COSMO (Python interface version 0.8.8 for Julia language), and MOSEK (Python interface). The sGS\_PADMM framework is presented as follows:
\begin{algorithm}[H]
\caption{sGS\_PADMM}
\label{algsGS}
\begin{algorithmic}[]
\STATE Let $\sigma>0, $ $\mu_{1}>0, $ $\mu_{2}>0,$ $\mu_{3}>0,$ and $\tau\in(0, (1+\sqrt{5})/2)$ be given parameters. Set $u^{0}=(v_{\Lambda}^{0},\xi^{0},w_{s}^{0})\in\mathcal{U} $ to be the initial point.\\
{$k=0,1,...,$}
\begin{align}
\xi^{k+1}&=\underset{\xi\in \mathcal{Z}\times\mathcal{Y} }{\mathrm{argmin}}~\mathcal{L}_\sigma ({v}_{\Lambda}^{k},\xi, {w}_{s}^{k}) + \frac{1}{2} \|{\xi}  - \xi^{k} \|^2_{\textrm{sGS}({\mathcal{Q}})+ \tilde{\mathcal{S}}'_0}\notag\\
v_{\Lambda}^{k+1}&=\underset{{v}_{\Lambda}\in \mathcal{V} }{\mathrm{argmin}}~\mathcal{L}_\sigma (v_{\Lambda},{\xi}^{k}, {w}_{s}^{k})  + \frac{1}{2} \|{v}_{\Lambda} - v^k_{\Lambda} \|^2_{{{\mu_{1} \mathcal{I}_{\mathcal{V}}}}}\notag\\
w_{s}^{k+1}&= w_{s}^k + \sigma\tau (\tilde{\mathcal{A}}^*{\xi}^{k+1} + v_{\Lambda}^{k+1}  -\tilde{r} )\notag
\end{align}
\end{algorithmic}
\end{algorithm}

\quad In the numerical experiments, SCS, COSMO, and MOSEK adopt the default settings (refer to the COSMO \cite{GarCOSMO}, MOSEK \cite{mosek}, and SCS \cite{Parikh} manuals for details). For APADMM, GPADMM, and sGS\_PADMM, we utilize the \verb"LDL"$^{\top}$ factorization \cite{SprseLDL} or the \verb"mkl-pardiso" function\footnote{\url{https://www.intel.com/content/www/us/en/docs/onemkl/get-started-guide/2023-0/overview.html}} to solve the linear systems from equations \eqref{bigslin} or \eqref{equ32}. The updates of $\Lambda$ in \eqref{ulambda} and $v$ in \eqref{upv} are obtained by using eigenvalue decomposition in \verb"MKL" for projection calculation. Regarding parameter settings, we choose $\mu_0=\mu_1=\mu_2=\mu_3=1e$-4 for these three methods, $\rho=2$ for APADMM, $\rho=1.8$ for GPADMM, and $\tau=1.618$ for sGS\_PADMM. We also adopt the strategy from \cite{2018samga} to adjust the penalty parameter $\sigma$. In addition, APADMM, GPADMM, and sGS\_PADMM employ the following stopping criterion based on the KKT relative residuals, which is similar to the ones used in SCS.
\begin{equation*}
  \mathrm{Err\_rel}:=\max\left\{\mathrm{p\_res},\mathrm{d\_res},\mathrm{\eta_{gap}}\right\},
\end{equation*}
where $\mathrm{d\_res}:=\max\{\eta_{s}^{k},\eta_{eq}^{k}\}$ with
\begin{align}
\eta_{s}^{k}=\frac{\|v^{k}+z^{k}\|_{\infty}}{1+\max\left\{\|v^{k}\|_{\infty},\|z^{k}\|_{\infty}\right\}}, \quad
\eta_{eq}^{k}=\frac{\|A_{w}^{\top}{z}^{k}+B_{w}^{\top}{y}^{k}+\Lambda^{k}-r\|_{\infty}}{1+\max\left\{\|A_{w}^{\top}{z}^{k}\|_{\infty},\|B_{w}^{\top}{y}^{k}\|_{\infty},\|\Lambda^{k}\|_{\infty},\|r\|_{\infty}\right\}},\notag
\end{align}
and $\mathrm{p\_res}:=\max\left\{\eta_{z}^{k},\eta_{\Lambda}^{k},\eta_{y}^{k},\underset{i=1,\ldots,M}{\max}\{\eta_{v_{i}}^{k}\}\right\}$ with
\begin{align}
\eta_{\Lambda}^{k}&=\frac{\|\Lambda^{k}-\Pi_{\Gamma^\frac{p(p+1)}{2}}(\Lambda^{k}-w^{k})\|_{\infty}}{1+\max\left\{\|\Lambda^{k}\|_{\infty},\|w^{k}\|_{\infty}\right\}}, \hspace{1.65cm}\eta_{y}^{k}=\frac{\|B_{w}w^{k}\|_{\infty}}{1+\|B_{w}w^{k}\|_{\infty}},\notag\\
\eta_{z}^{k}&= \frac{\|b_{w}+A_{w}w^{k}+s^{k}\|_{\infty}}{1+\max\left\{\|b_{w}\|_{\infty},\|s^{k}\|_{\infty},\|A_{w }w^{k}\|_{\infty}\right\}},\quad\eta_{v_{i}}^{k}=\frac{\|v_{i}^{k}-\Pi_{\Gamma^\frac{n(n+1)}{2}}(v_{i}^{k}-s_{i}^{k})\|_{\infty}}{1+\max\left\{\|v_{i}^{k}\|_{\infty},\|s_{i}^{k}\|_{\infty}\right\}}.\notag
\end{align}
Finally, the relative gap $\mathrm{\eta_{gap}}$ is given by
\begin{align}
\mathrm{\eta_{gap}}:=\frac{|\mathrm{p\_obj}-\mathrm{d\_obj}|}{1+\max\left\{|\mathrm{p\_obj}|,|\mathrm{d\_obj}|\right\}},\notag
\end{align}
where 
\begin{equation*}
\mathrm{p\_obj}:= \langle r,w^k\rangle,\quad \mathrm{d\_obj}:=-\langle z^k,b_{w}\rangle.
\end{equation*}

\subsection{Verification of decentralized structure}
In this subsection, taking the chemical reactor system from \cite{Dectest} as an example, we verify that the feedback gain matrix $K$ constructed from the solution obtained by Algorithm \ref{alg4} satisfies the decentralized structure defined in \eqref{a3}. We also test the stability of the control.  In this experiment, the estimated matrices $(A,B_{2})$ and the given matrices $(B_{1},C,D)$ in the LTI system \eqref{a1} are provided by:
\begin{align*}
 A &= \left[
  \begin{array}{cccc}
    -1.38  & -0.2077 & 6.715  &-5.676\\
    -0.5814&-4.29& 0& 0.675\\
    1.067&4.273&-6.654&5.893\\
    -0.048&-4.273& 1.343&-2.104\\
  \end{array}
\right],~ B_{1}=\left[
                 \begin{array}{cccc}
                   1 & 0 & 0 &0 \\
                   0 & 1 & 0 &0\\
                   0 & 0 & 1 &0\\
                   0 & 0 & 0 &1\\
                 \end{array}
               \right],
\\
  B_{2} & =\left[
             \begin{array}{cc}
               0,   &  0 \\
               5.679&  0 \\
              1.136 &-3.146 \\
              1.136 &  0\\
             \end{array}
           \right],~ C=\left[
                        \begin{array}{cccc}
                          0&1&0&0\\
                          0&0&1&0\\
                          0&0&0&0\\
                          0&0&0&0\\
                        \end{array}
                      \right],~ D = \left[
                                     \begin{array}{cc}
                                       0 & 0 \\
                                       0 & 0 \\
                                       1 & 0 \\
                                       0 & 1 \\
                                     \end{array}
                                   \right].
\end{align*}
Furthermore, we introduce noise to the system matrix $A$ with a magnitude of $\pm5\%$ of its nominal values and consider $M=4$ as the number of extreme systems. The solution $W$ obtained by APADMM with $\mathrm{Err\_rel}\leq 1e$-7 is given by 

\begin{equation*}
W=\left[
    \begin{array}{cccc|cc}
      1.94008 & -0.15198 & 0        & 0         & 0.05891    &  0 \\
      -0.15198& 0.09056  & 0        & 0         & 0.05654    &  0 \\
      0       & 0        & 0.30031  & 0.26107   & 0          & -0.17958\\
      0       & 0        & 0.26107  & 0.47226   & 0          & -0.23248 \\\hline
      0.05891 & 0.05654  & 0        & 0         & 0.04935    & 0 \\
      0       & 0        & -0.17958 & -0.23248  & 0          & 0.13116\\
    \end{array}
  \right].
\end{equation*}
Then, the feedback gain matrix $K=W_{2}^{\top}W_{1}^{-1}$ is
\begin{equation}\label{decenK}
K=W_{2}^{\top}W_{1}^{-1}=\left[
                           \begin{array}{cc|cc}
                             0.09128 & 0.77760 & 0 & 0 \\\hline
                             0 & 0 & -0.32737 & -0.31129 \\
                           \end{array}
                         \right],
\end{equation}
which satisfies the decentralized structure. In the simulation, $w(t)$ is characterized as a vector of the impulse disturbance. The responses of all state variables are illustrated in Fig.\ref{fig1}.
\begin{figure}[H]
  \centering
  \includegraphics[width=6.3cm]{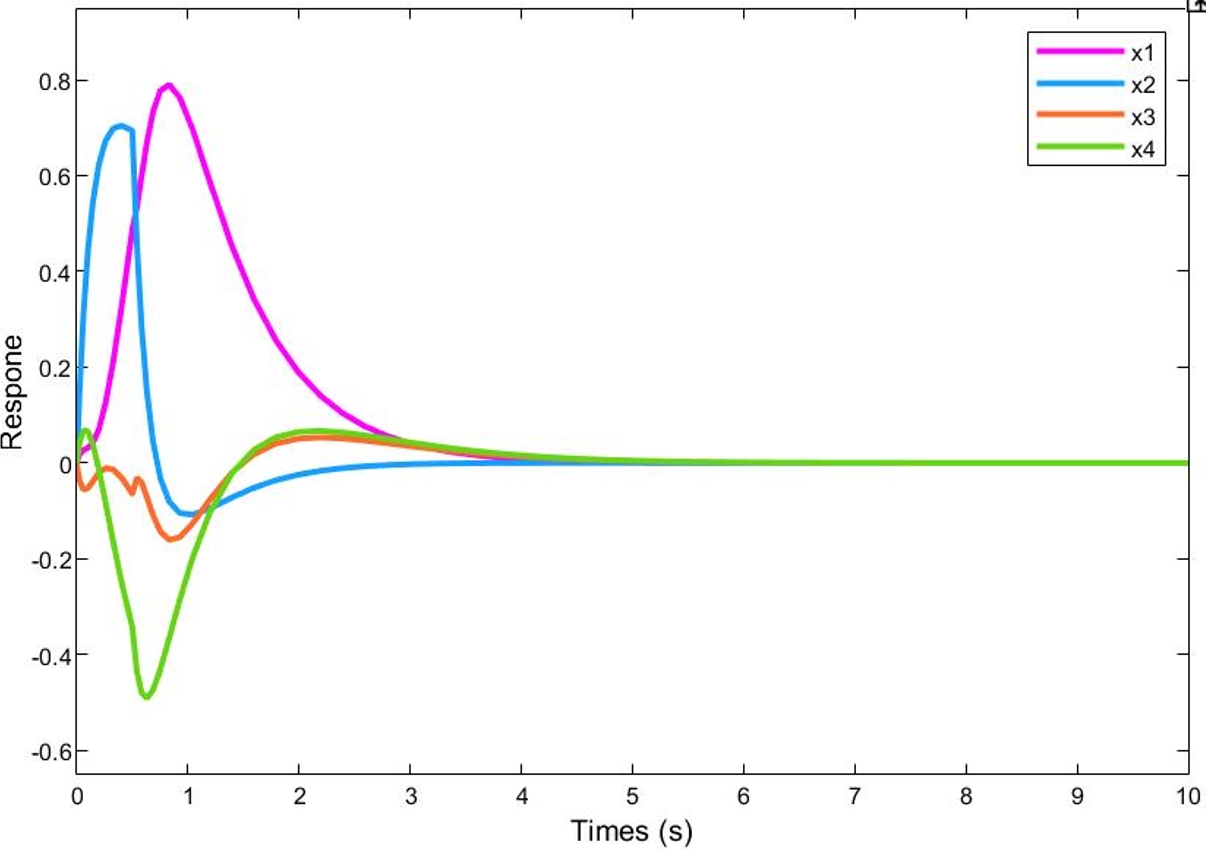}
  \caption{System response under the feedback matrix $K$.}\label{fig1}
\end{figure}
\ni From the system response in Fig.\ref{fig1}, it can be seen that the robust stability is guaranteed for the LTI system.

\subsection{Numerical results for the lifting technique and acceleration}

In this subsection, we will first demonstrate the superiority of the lifting technique for solving large-scale sparse linear systems instead of solving the normal equation \eqref{bigs} directly. In particular, the numerical results of the lifting technique are presented in Table \ref{tablift1}, where PADMM\_AA$^{\top}$ denotes the result of solving the normal equation \eqref{bigs} by $\verb"LDL"^{\top}$ decomposition directly, and PADMM\_Lifting denotes the result of the lifting technique.
\begin{table}[!htbp]
\centering
\caption{The performance comparison between PADMM\_$\mathrm{AA}^{\top}$ and PADMM\_Lifting with Err\_rel=1e-5.}\label{tablift1}
\begin{tabular}{cccccccc}
\toprule
Algorithm                      &n, m, M        &Iter      &Err\_rel      &Time (s)               &Lin\_sys Time (s)\footnotemark[1] \\
\hline
test 1 \\
PADMM\_$\mathrm{AA}^{\top}$    &(10, 5, 11)    &993        &4.994e-06    &3.54e-01              &1.59e-01    \\
PADMM\_Lifting                 &(10, 5, 11)    &941        &8.273e-06    &1.87e-01             &1.59e-02    \\
\hline
test 2\\
PADMM\_$\mathrm{AA}^{\top}$    &(30, 15, 11)   &1112       &9.698e-06    &5.22e+01              &2.94e+01    \\
PADMM\_Lifting                 &(30, 15, 11)   &1353       &2.521e-06    &4.55e+00             &1.44e+00    \\
\hline
test 3\\
PADMM\_$\mathrm{AA}^{\top}$    &(50, 25, 11)   &1202       &2.674e-06    &6.54e+02              &2.42e+02    \\
PADMM\_Lifting                 &(50, 25, 11)   &1151       &5.278e-06    &2.13e+01             &9.55e+00    \\
\bottomrule
\end{tabular}
\footnotetext[1]{Lin\_sys Time refers to the total time spent on solving the linear system.}
\end{table}From Table \ref{tablift1}, we observe that the lifting technique is more than 10 times faster than directly solving the normal equation \eqref{bigs}, and the acceleration effect improves with the increase in problem size. The possible reason is that $\tilde{\mathcal A}\tilde{\mathcal A}^{*}$ turns out to be a dense matrix even if $\tilde{\mathcal A}$ is a sparse matrix, and the complexity of the $\verb"LDL"^{\top}$ decomposition is directly related to sparsity {\cite[Chapter 4]{Davis}}. In contrast, the lifting technique can avoid calculating $\tilde{\mathcal A}\tilde{\mathcal A}^{*}$ and fully exploit the sparsity of $\tilde{\mathcal A} $.

\quad According to the numerical testing, the restart strategy \cite[Section 11.4]{Nemirov} and \cite[Section 5.1]{Neste} prove to be highly beneficial in enhancing the performance of APADMM. To determine an appropriate fixed iteration number for restarting the APADMM algorithm, we randomly generate 30 examples and experiment with various restart intervals to solve problem \eqref{DP3}. As illustrated in Fig. \ref{figrestartp}, setting the restart interval for APADMM with the operator $\textrm{sGS}({\mathcal{Q}})+ \tilde{\mathcal{S}}'_0$ to 58 results in significantly lower average and median iteration numbers. Conversely, for APADMM with the operator $\tilde{\mathcal{S}}'_0$, an optimal restart interval is found to be 18.
\begin{figure}[H]    
\subfloat[APADMM with operator $\textrm{sGS}({\mathcal{Q}})+ \tilde{\mathcal{S}}'_0${\centering}]{\includegraphics[width=0.5\textwidth, height= 5.cm]{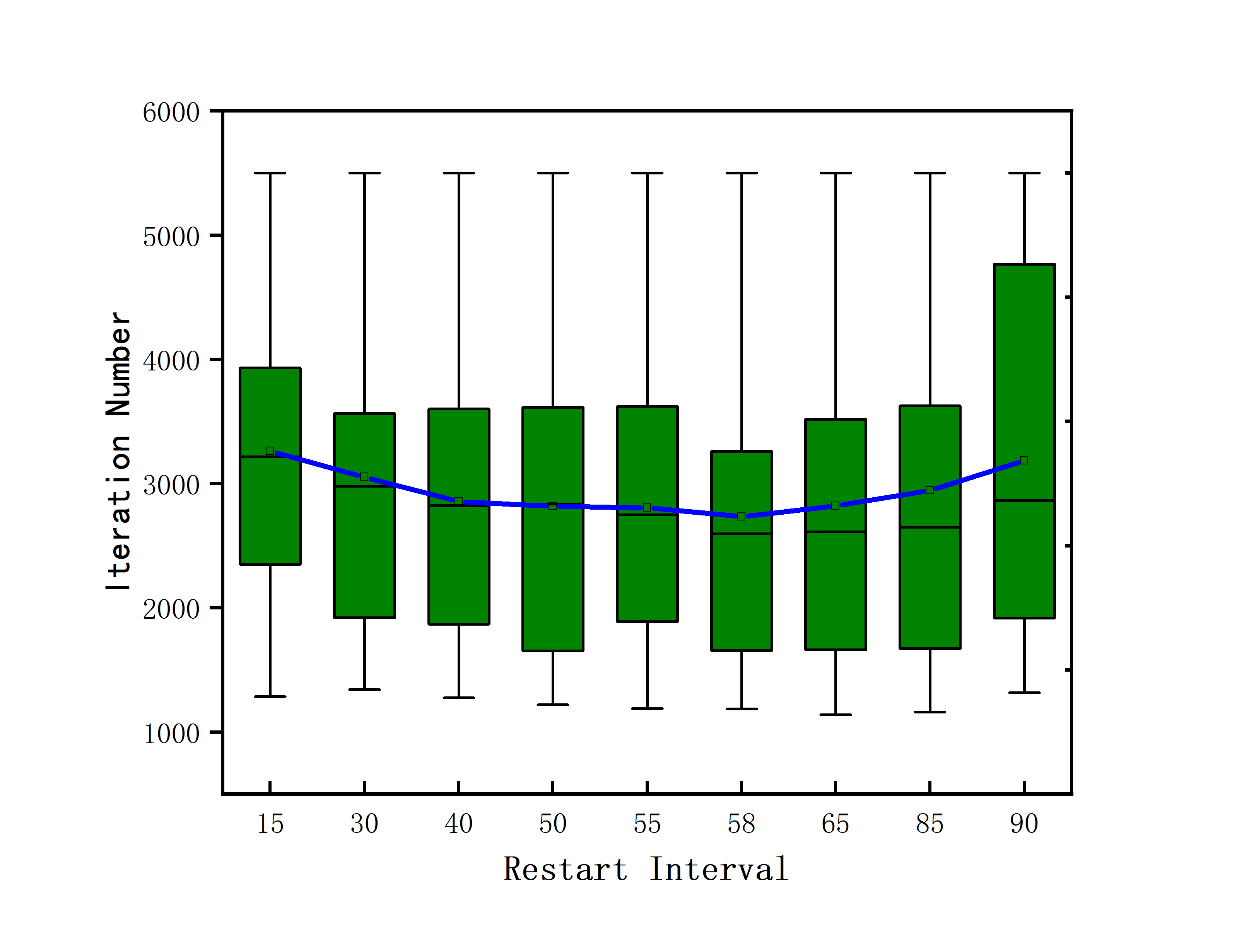}}         
\subfloat[APADMM with operator $\tilde{\mathcal{S}}_0${\centering}]{\includegraphics[width=0.5\textwidth, height= 5.cm]{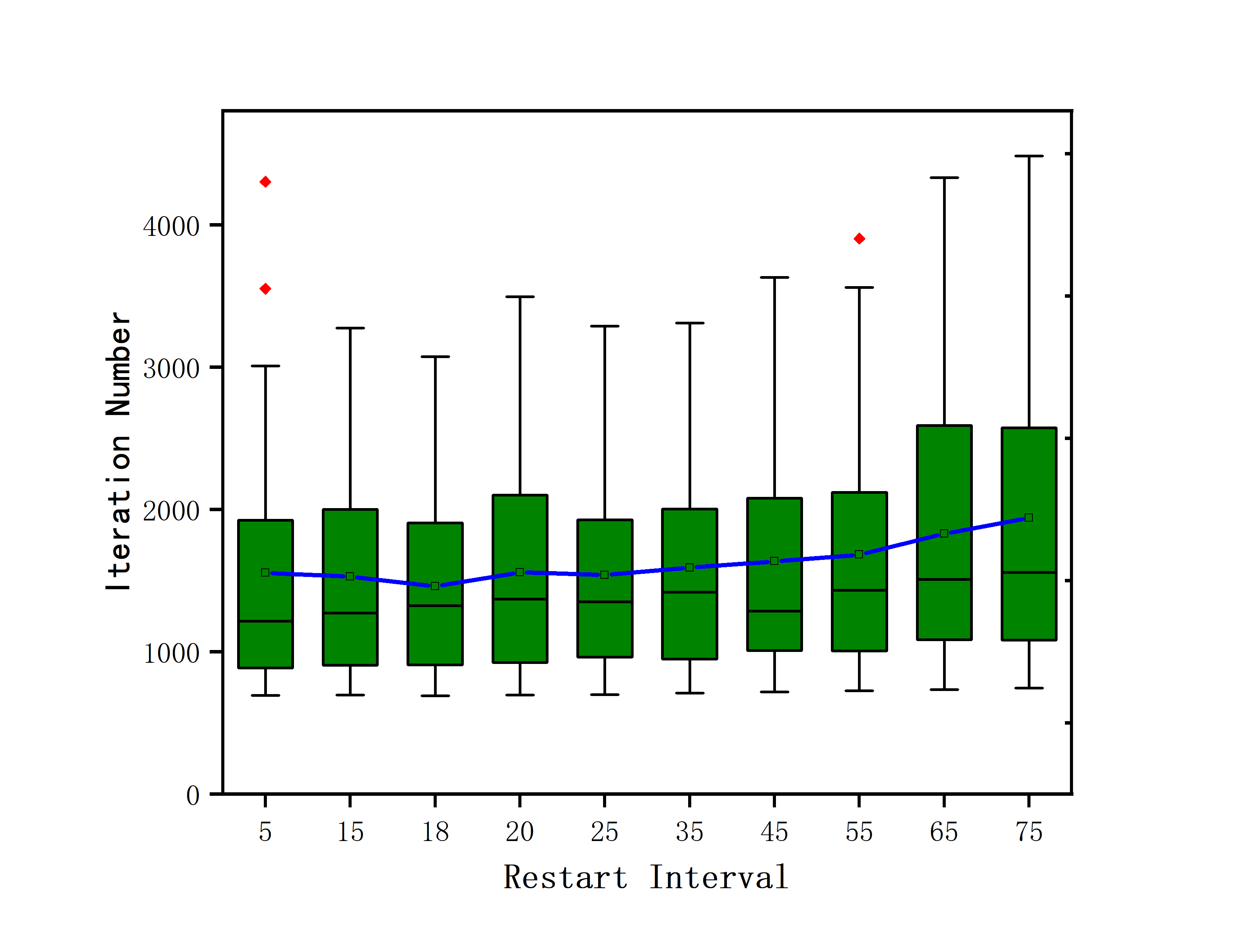}\label{figreb}}  
\caption{The performance of APADMM with various restart intervals.}\label{figrestartp}
\end{figure}
After setting the restart interval, we compare APADMM with some classical PADMM-type methods such as GPADMM and sGS\_PADMM. Both APADMM and GPADMM utilize the proximal operator $\textrm{sGS}({\mathcal{Q}})+ \tilde{\mathcal{S}}'_0$. Fig.\ref{figp2} shows that APADMM can achieve a better solution in terms of Err\_rel with fewer iterations than GPADMM and sGS\_PADMM for different problem scales. And, GPADMM exhibits comparable performance to sGS\_PADMM. 
\begin{figure}[!htbp]
	\centering
	\begin{subfigure}{0.325\linewidth}
		\centering
		\includegraphics[width=1.1\linewidth]{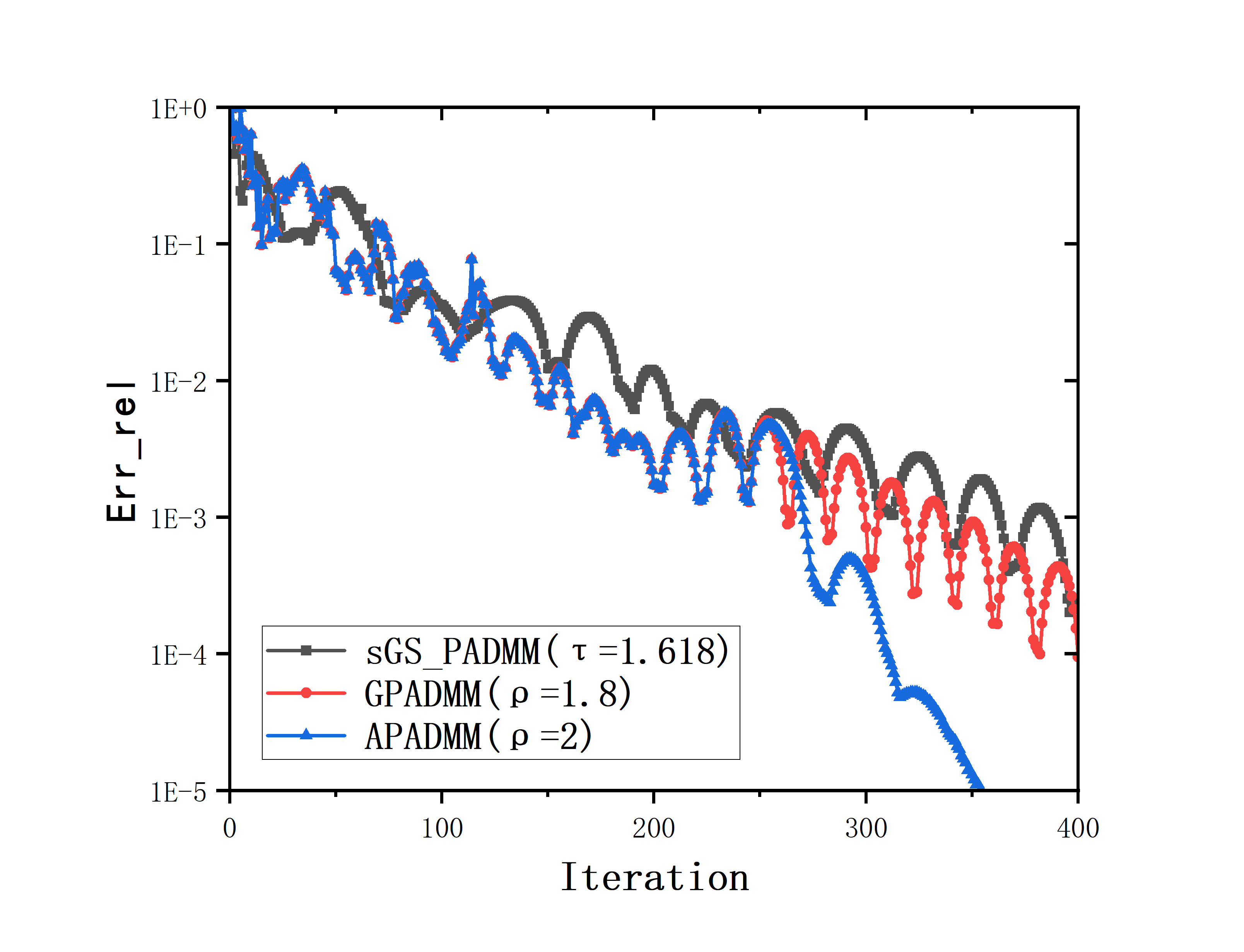}
		\caption{n=10, m=3, M=6}
		\label{chutian1}
	\end{subfigure}
	\centering
	\begin{subfigure}{0.325\linewidth}
		\centering
		\includegraphics[width=1.1\linewidth]{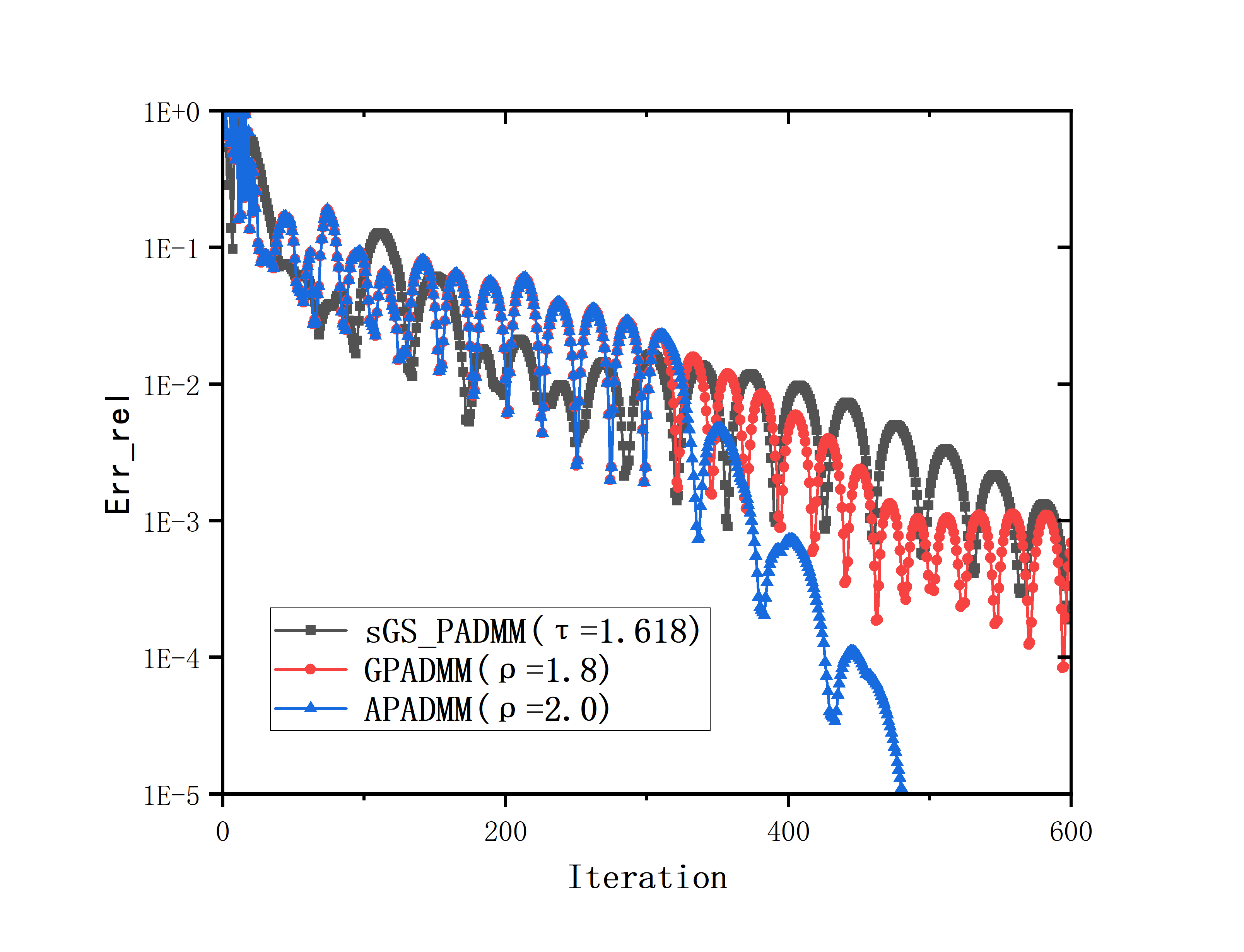}
		\caption{n=12, m=3, M=6}
		% \label{chutian2}
	\end{subfigure}
	\centering
	\begin{subfigure}{0.325\linewidth}
		\centering
		\includegraphics[width=1.1\linewidth]{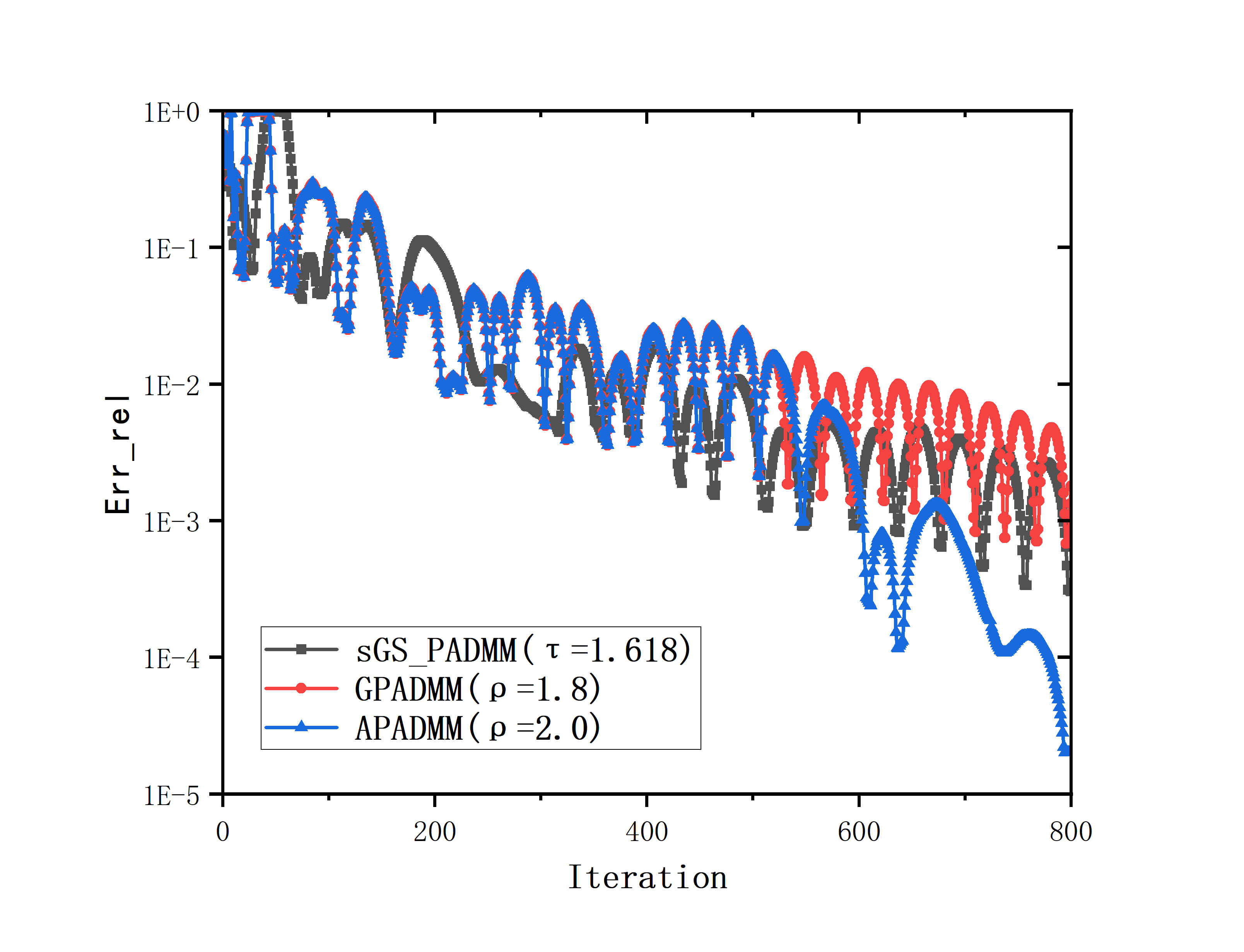}
		\caption{n=12, m=6, M=8}
	\end{subfigure}
	\centering
	\begin{subfigure}{0.325\linewidth}
		\centering
		\includegraphics[width=1.1\linewidth]{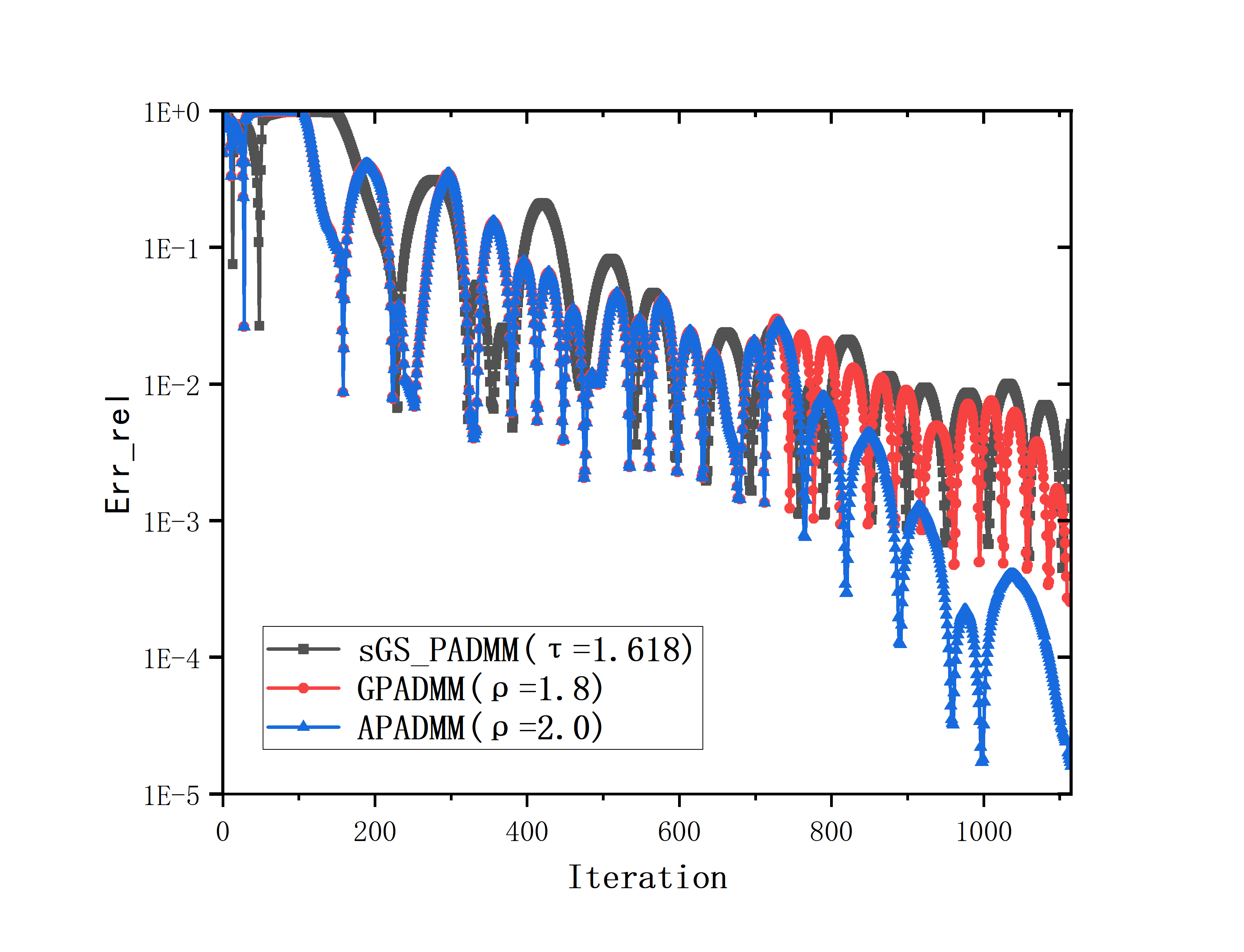}
		\caption{n=15, m=15, M=10}
		% \label{chutian2}
	\end{subfigure}
	\centering
	\begin{subfigure}{0.325\linewidth}
		\centering
		\includegraphics[width=1.1\linewidth]{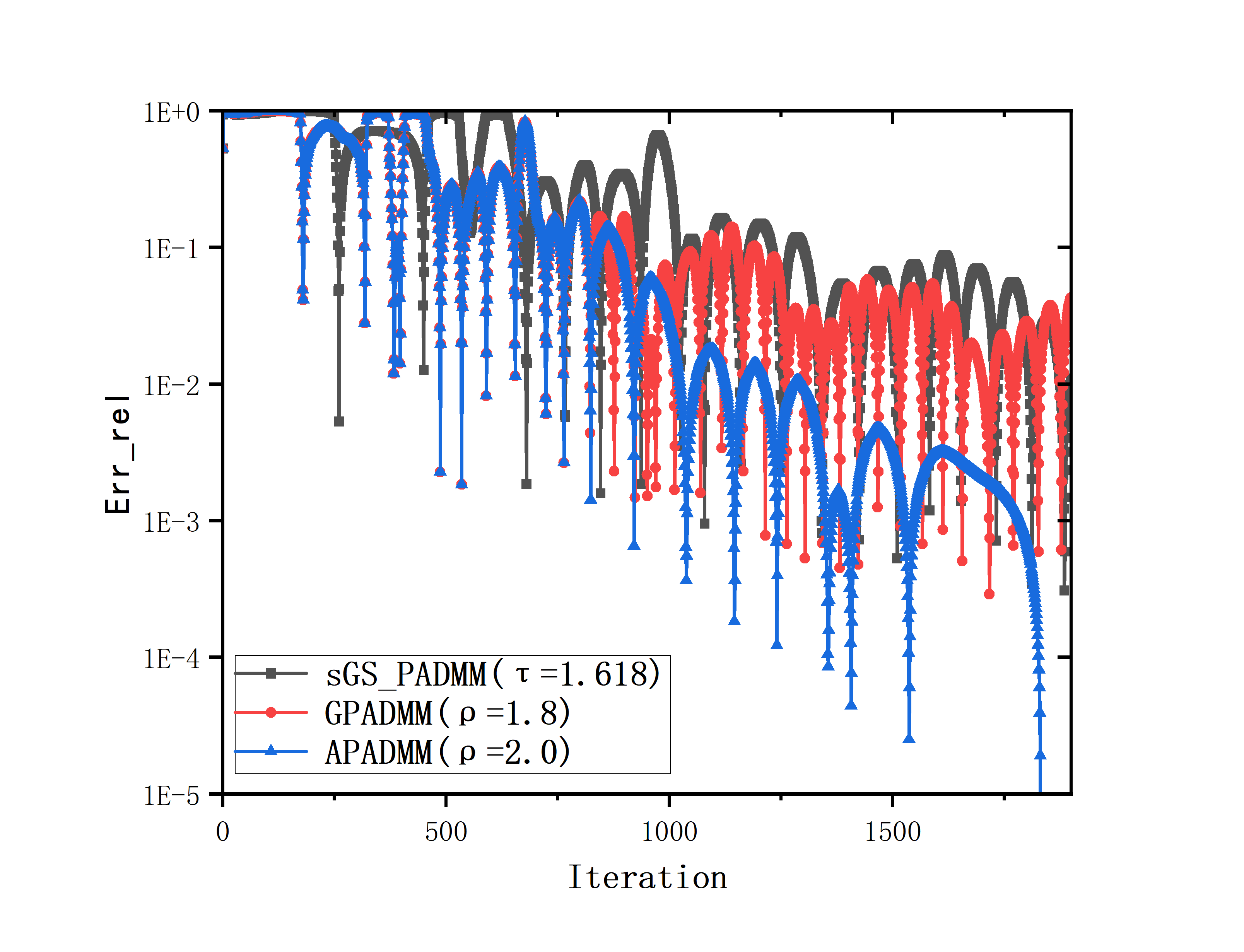}
		\caption{n=20, m=15, M=10}
	\end{subfigure}
	\centering
	\begin{subfigure}{0.325\linewidth}
		\centering
		\includegraphics[width=1.1\linewidth]{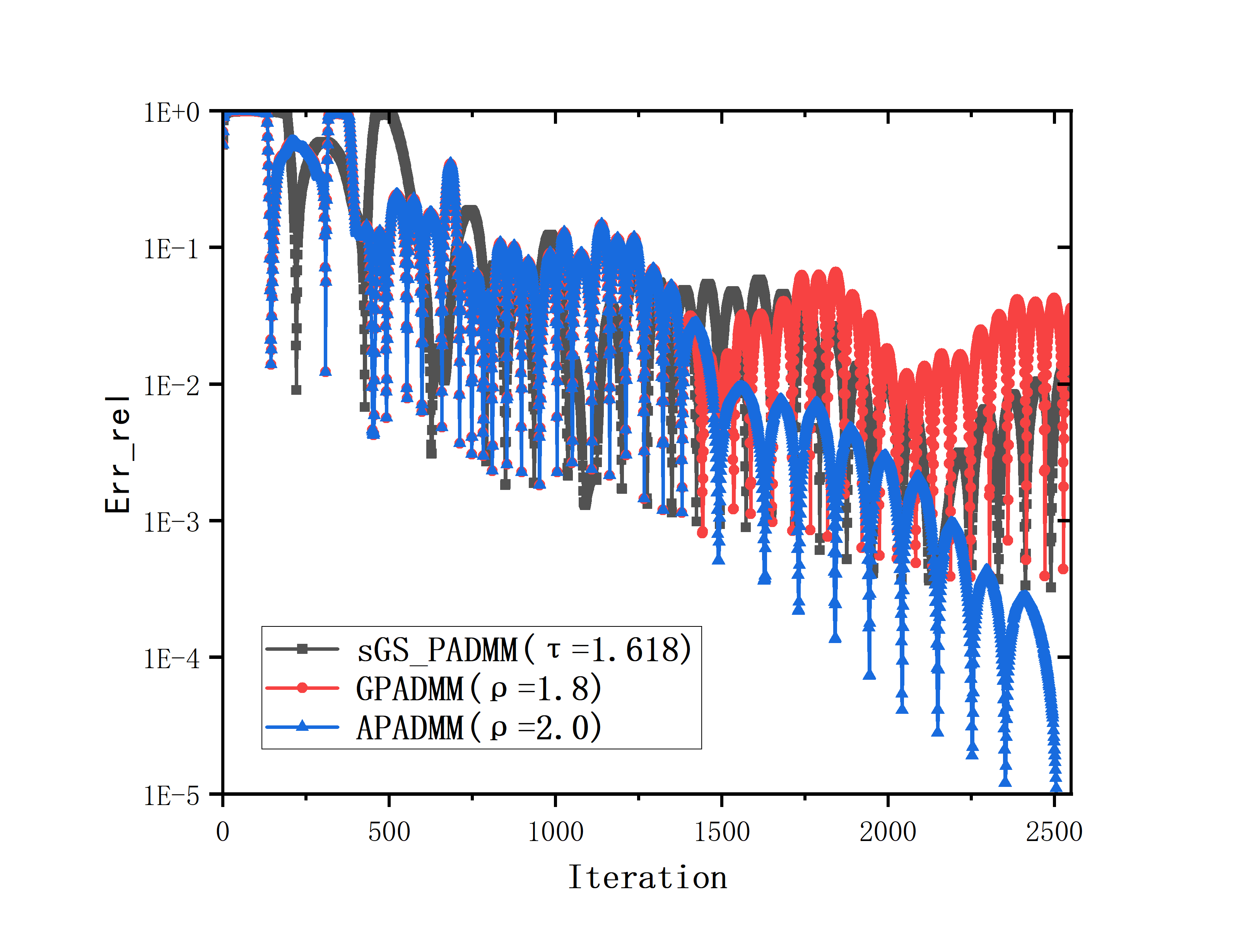}
		\caption{n=30, m=20, M=15}
	\end{subfigure}
\caption{The performance comparison of APADMM, sGS\_PADMM, and GPADMM.}\label{figp2}
\end{figure}
To delve deeper into the impact of different proximal operators on acceleration, we summarized some numerical results of APADMM and PADMM in Table \ref{Tabsum2}, and specific experimental results for each example can be found in Table \ref{Tabacc4}. Here, ``TB" denotes the proximal operator $\tilde{\mathcal{S}}_0$, and ``sGS" denotes the proximal operator $\textrm{sGS}({\mathcal{Q}})+ \tilde{\mathcal{S}}'_0$. As illustrated in Table \ref{Tabsum2}, sGS\_APADMM demonstrates an average iteration step improvement of {\bf{46.268}$\%$} and reduces the average time by {\bf{47.378}$\%$} compared to sGS\_GPADMM. Additionally, TB\_APADMM shows an average iteration step improvement of {\bf{25.668}$\%$} and a time reduction of {\bf{26.809}$\%$} compared to TB\_GPADMM. These results emphasize that the acceleration technique significantly enhances the efficiency of PADMM. Moreover, the impact of the acceleration technique on the proximal operator $\textrm{sGS}({\mathcal{Q}})+ \tilde{\mathcal{S}}'_0$ is more pronounced than on the proximal operator $\tilde{\mathcal{S}}_0$. One possible explanation is that the accelerated effect becomes more significant with larger iteration numbers, and PADMM with $\textrm{sGS}({\mathcal{Q}})+ \tilde{\mathcal{S}}'_0$ involves more iterations. 

\begin{table}[]
\caption{The comparison of the acceleration effects on proximal operators $\tilde{\mathcal{S}}_0$ and $\textrm{sGS}({\mathcal{Q}})+ \tilde{\mathcal{S}}'_0$.
}\label{Tabsum2}
	\centering
	\begin{tabular}{cccccccccc}
		\toprule
		\multicolumn{1}{c}{\multirow{2}*{Method}}& \multicolumn{4}{c}{Iteration Number} & \multicolumn{4}{c}{CUP Times (s)} \\
		\cmidrule(lr){2-5}\cmidrule(lr){6-9}
		                & Avgs    & Med.    &  Max.    & Min.    & Avgs     & Med.     &  Max.    & Min. \\
		\midrule
		TB\_GPADMM      & 1706.26 & 1503    & 4193     & 617     & 2.35e+02 &5.91e+01  & 3.85e+03 & 2.44e-01 \\
        TB\_APADMM      & 1268.30 & 1197    & 2604     & 572     & 1.72e+02 &3.66e+01  & 2.38e+03 & 9.67e-02 \\
        sGS\_GPADMM     & 5599.54 & 5152    & 13628    & 2201    & 5.34e+02 &1.30e+02  & 7.86e+03 & 5.20e-01 \\
        sGS\_APADMM     & 3008.76 & 3010    & 6210     & 1298    & 2.81e+02 &8.92e+01  & 3.65e+03 & 2.40e-01 \\ 
		\bottomrule
	\end{tabular}
\end{table}

\begin{sidewaystable}
\begin{centering}
        \caption{Comparison of the acceleration performance of APADMM with different proximal operators for the ODC problems with Err\_rel=1e-5.}\label{Tabacc4}
		\begin{spacing}{1.15}
			\begin{tabular}{|ccc|ccc|ccc|ccc|ccc|}
			\hline
\multirow{2}{*}{\makecell[{}{p{0.2cm}}]{\textbf{n}}}&\multirow{2}{*}{\makecell[{}{p{0.2cm}}]{\textbf{m}}}&\multirow{2}{*}{\makecell[{}{p{0.2cm}}]{\textbf{M}}} &\multicolumn{3}{c|}{{TB\_GPADMM}}&\multicolumn{3}{c|}{{TB\_APADMM}}&\multicolumn{3}{c|}{{sGS\_GPADMM}}&\multicolumn{3}{c|}{{sGS\_APADMM}}\\
                \cmidrule(r){4-15}
                &&&\makecell[{}{p{0.5cm}}]{Iter}& \makecell[{}{p{1.3cm}}]{Err\_rel}& \makecell[{}{p{1.3cm}}]{Time (s)}
                &\makecell[{}{p{0.5cm}}]{Iter}  & \makecell[{}{p{1.3cm}}]{Err\_rel}& \makecell[{}{p{1.3cm}}]{Time (s)}
                &\makecell[{}{p{0.5cm}}]{Iter}  & \makecell[{}{p{1.3cm}}]{Err\_rel}& \makecell[{}{p{1.3cm}}]{Time (s)}
                &\makecell[{}{p{0.5cm}}]{Iter}  & \makecell[{}{p{1.3cm}}]{Err\_rel}& \makecell[{}{p{1.3cm}}]{Time (s)}\\
				\hline
                 10  &10  &5  &2674&9.807e-06&4.40e-01 &692 &9.827e-06&9.67e-02 &2837 &9.110e-06&5.20e-01 &1320 &8.900e-06&2.40e-01\\
                 10  &5   &10 &1167&4.019e-06&2.44e-01 &900 &9.936e-06&1.86e-01 &4245 &9.368e-06&8.66e-01 &1298 &9.948e-06&2.45e-01\\
			     15  &10  &10 &1639&4.267e-06&6.52e-01 &704 &7.701e-06&3.07e-01 &4654 &7.529e-06&2.27e+00 &3080 &9.879e-06&1.45e+00\\
                 15  &15  &10 &948 &5.776e-06&4.05e-01 &1422&8.151e-06&6.41e-01 &2231 &9.488e-06&1.19e+00 &1927 &9.940e-06&9.82e-01\\
                 20  &10  &10 &2489&9.500e-06&1.51e+00 &887 &7.659e-06&5.23e-01 &4832 &9.440e-06&3.15e+00 &1416 &8.511e-06&9.06e-01\\
                 20  &15  &10 &980 &9.033e-06&9.36e-01 &600 &5.024e-06&5.76e-01 &4612 &9.301e-06&4.36e+00 &1561 &9.957e-06&1.50e+00\\
                 25  &20  &15 &2753&9.123e-06&4.61e+00 &1466&9.570e-06&2.19e+00 &5778 &1.285e-06&8.80e+00 &2597 &9.586e-06&3.99e+00\\
                 30  &20  &10 &676 &5.876e-06&1.77e+00 &755 &4.145e-06&1.73e+00 &6752 &9.414e-06&1.39e+01 &3741 &6.331e-06&7.97e+00\\
                 30  &20  &15 &617 &6.577e-06&1.28e+00 &572 &9.101e-06&9.99e-01 &5543 &2.885e-06&9.19e+00 &1964 &9.097e-06&3.32e+00\\
                 35  &20  &10 &1499&9.355e-06&4.03e+00 &961 &9.865e-06&2.46e+00 &5321 &5.865e-06&1.42e+01 &3575 &9.896e-06&8.54e+00\\
                 35  &25  &15 &1530&4.502e-06&6.54e+00 &1144&8.773e-06&4.41e+00 &8117 &7.063e-06&3.28e+01 &6210 &5.171e-06&2.43e+01\\
                 35  &35  &10 &689 &8.391e-06&2.95e+00 &827 &8.409e-06&3.13e+00 &7262 &4.548e-06&2.78e+01 &2752 &4.963e-06&1.12e+01\\
                 40  &10  &10 &747 &9.002e-06&2.04e+00 &834 &9.492e-06&2.11e+00 &3200 &1.250e-06&8.43e+00 &1381 &9.404e-06&4.50e+00\\
                 40  &20  &10 &620 &9.800e-06&2.06e+00 &704 &9.663e-06&2.10e+00 &4173 &2.117e-06&1.36e+01 &1656 &6.838e-06&6.18e+00\\
                 40  &20  &15 &1503&8.902e-06&8.05e+00 &1029&8.648e-06&5.10e+00 &10477&1.743e-06&4.75e+01 &3801 &8.420e-06&2.11e+01\\
                 40  &20  &20 &1154&9.547e-06&7.89e+00 &1242&9.883e-06&7.47e+00 &7689 &3.400e-06&4.84e+01 &2575 &5.132e-06&1.62e+01\\
                 40  &25  &15 &2176&9.616e-06&1.28e+01 &2089&9.994e-06&1.13e+01 &13628&5.546e-06&6.99e+01 &5249 &6.896e-06&3.13e+01\\
                 50  &20  &10 &1551&9.444e-06&1.11e+01 &1197&9.954e-06&2.26e+01 &6611 &5.208e-06&4.61e+01 &2855 &9.896e-06&2.02e+01\\
                 60  &30  &10 &1235&7.927e-06&1.65e+01 &1831&6.135e-06&2.05e+01 &5199 &7.484e-06&6.59e+01 &2898 &8.928e-06&3.75e+01\\
                 60  &35  &10 &974 &9.119e-06&1.57e+01 &1120&9.896e-06&1.63e+01 &7666 &1.148e-06&1.21e+02 &4524 &8.759e-06&7.06e+01\\
                 70  &10  &30 &1973&3.783e-06&5.91e+01 &1623&4.046e-06&5.04e+01 &11353&7.577e-06&3.47e+02 &4069 &4.447e-06&1.18e+02\\
                 80  &10  &10 &879 &7.636e-06&1.87e+01 &910 &7.944e-06&1.81e+01 &3991 &2.497e-06&7.89e+01 &3433 &4.377e-06&1.63e+02\\
                \hline
			\end{tabular}
		\end{spacing}
	\end{centering}
\end{sidewaystable}

\begin{sidewaystable}
\centering
		\begin{spacing}{1.15}
			\begin{tabular}{|ccc|ccc|ccc|ccc|ccc|}
			\hline
\multirow{2}{*}{\makecell[{}{p{0.2cm}}]{\textbf{n}}}&\multirow{2}{*}{\makecell[{}{p{0.2cm}}]{\textbf{m}}}&\multirow{2}{*}{\makecell[{}{p{0.2cm}}]{\textbf{M}}} &\multicolumn{3}{c|}{{TB\_GPADMM}}&\multicolumn{3}{c|}{{TB\_APADMM}}&\multicolumn{3}{c|}{{sGS\_GPADMM}}&\multicolumn{3}{c|}{{sGS\_APADMM}}\\
                \cmidrule(r){4-15}
                &&&\makecell[{}{p{0.5cm}}]{Iter}& \makecell[{}{p{1.3cm}}]{Err\_rel}& \makecell[{}{p{1.3cm}}]{Time (s)}
                &\makecell[{}{p{0.5cm}}]{Iter}  & \makecell[{}{p{1.3cm}}]{Err\_rel}& \makecell[{}{p{1.3cm}}]{Time (s)}
                &\makecell[{}{p{0.5cm}}]{Iter}  & \makecell[{}{p{1.3cm}}]{Err\_rel}& \makecell[{}{p{1.3cm}}]{Time (s)}
                &\makecell[{}{p{0.5cm}}]{Iter}  & \makecell[{}{p{1.3cm}}]{Err\_rel}& \makecell[{}{p{1.3cm}}]{Time (s)}\\
                \hline
                100  &6  &8   &843 &4.347e-06&2.85e+01 &1104&9.990e-06&3.37e+01 &3394&8.479e-06&9.33e+01  &2104&9.525e-06&6.40e+01\\
                110  &3  &9   &2037&7.834e-06&9.49e+01 &1186&9.135e-06&5.22e+01 &6366&4.250e-06&2.39e+02  &3448&6.028e-06&1.45e+02\\
                115  &4  &8   &1534&5.638e-06&7.92e+01 &1107&9.633e-06&5.51e+01 &3968&9.229e-06&1.69e+02  &4058&4.513e-06&1.84e+02\\
                120  &2  &4   &1655&9.518e-06&6.54e+01 &975 &8.984e-06&3.66e+01 &3969&8.384e-06&1.30e+02  &2438&6.211e-06&8.80e+01\\
                120  &2  &5   &2909&9.284e-06&1.23e+02 &1599&8.804e-06&6.76e+01 &3891&7.784e-06&1.44e+02  &2343&8.654e-06&8.92e+01\\
                120  &6  &10  &1320&7.219e-06&8.06e+01 &1499&9.713e-06&9.13e+01 &4793&8.541e-06&2.52e+02  &3260&3.264e-06&1.76e+02\\
                125  &3  &9   &1667&6.999e-06&1.06e+02 &1649&8.362e-06&1.06e+02 &5152&6.185e-06&2.94e+02  &4373&4.220e-06&2.71e+02\\
                125  &4  &8   &1645&6.946e-06&1.02e+02 &1376&6.567e-06&8.73e+01 &4580&7.452e-06&2.54e+02  &3295&8.380e-06&1.86e+02\\
                125  &8  &10  &1300&9.097e-06&9.79e+01 &1477&9.526e-06&1.11e+02 &6139&9.097e-06&4.05e+02  &3135&6.535e-06&2.11e+02\\
                130  &2  &5   &4193&8.088e-06&2.25e+02 &1895&5.067e-06&1.02e+02 &8307&9.648e-06&4.01e+02  &3694&4.138e-06&1.82e+02\\
                130  &3  &9   &2007&5.607e-06&1.44e+02 &1512&7.701e-06&1.12e+02 &5274&5.579e-06&3.38e+02  &3217&9.377e-06&2.12e+02\\
                130  &6  &8   &1210&9.863e-06&9.09e+01 &1297&8.555e-06&9.84e+01 &4832&7.171e-06&3.11e+02  &2588&7.735e-06&1.73e+02\\
                135  &8  &10  &1368&9.240e-06&1.32e+02 &1481&9.772e-06&1.42e+02 &5498&8.629e-06&4.56e+02  &3639&4.945e-06&3.09e+02\\
                140  &2  &5   &2414&8.395e-06&1.71e+02 &1602&8.484e-06&1.17e+02 &4236&9.980e-06&2.67e+02  &3010&7.989e-06&1.94e+02\\ 
                140  &7  &10  &1225&9.203e-06&1.23e+02 &1361&9.372e-06&1.40e+02 &5295&9.403e-06&4.60e+02  &3670&7.849e-06&3.48e+02\\
                145  &2  &4   &2181&4.441e-06&1.60e+02 &1264&7.132e-06&9.43e+01 &3557&7.884e-06&2.32e+02  &2715&6.599e-06&1.93e+02\\
                145  &4  &8   &1446&2.780e-06&1.46e+02 &1246&7.344e-06&1.29e+02 &4668&7.445e-06&4.13e+02  &3958&7.208e-06&3.57e+02\\
                165  &3  &9   &2843&5.058e-06&4.37e+02 &2088&7.901e-06&3.24e+02 &5809&8.871e-06&8.03e+02  &3913&9.992e-06&5.71e+02\\
                185  &2  &4   &2237&7.741e-06&3.80e+02 &1143&9.136e-06&2.08e+02 &5135&8.788e-06&7.71e+02  &1684&9.096e-06&2.91e+02\\
                195  &2  &4   &2476&7.627e-06&5.01e+02 &1417&6.029e-06&3.03e+02 &5205&9.923e-06&9.54e+02  &3651&8.626e-06&7.25e+02\\
                215  &2  &4   &1479&9.923e-06&5.47e+02 &742 &1.035e-06&3.11e+02 &2972&6.812e-06&8.33e+02  &1339&7.248e-06&4.39e+02\\
                225  &2  &4   &1115&9.901e-06&4.42e+02 &827 &9.994e-06&3.58e+02 &2201&9.445e-06&7.98e+02  &2385&9.926e-06&8.77e+02\\
                275  &3  &9   &2949&5.388e-06&2.51e+03 &2604&5.643e-06&2.29e+03 &8069&8.693e-06&6.73e+03  &4329&9.324e-06&3.65e+03\\
                300  &2  &4   &3962&7.933e-06&3.85e+03 &2382&7.942e-06&2.38e+03 &8098&9.126e-06&7.86e+03  &2275&4.454e-06&2.42e+03\\  
                \hline
			\end{tabular}
		\end{spacing}
\end{sidewaystable}

\subsection{Numerical results for APADMM in comparison with COSMO, MOSEK, and SCS}

In this subsection, we compare APADMM with COSMO, MOSEK (the interior point method), and SCS in ODC problems of different scales. All solvers will be terminated with a tolerance of 1e-4. In the following tests, the maximum number of iterations is set to 25000 for APADMM, SCS, COSMO, and MOSEK for small and medium-scale problems, and to 10000 for large-scale problems. We summarize the results of small-scale, medium-scale, and large-scale problems in Table \ref{small}, \ref{medium}, and \ref{large} respectively, where the status \verb"max_iter" refers to the algorithm stopping due to reaching the maximum iteration steps, \verb"solved" means that the solver has successfully solved the ODC problems to a given accuracy, and \verb"memory error" indicates that the solver has reached its memory limits, resulting in an error. Moreover, the term ``obj\_bias" represents the absolute difference between the objective function values obtained by each algorithm and that of MOSEK. Additionally, ``p\_abs" and ``d\_abs" denote the absolute residual error, while ``p\_res" and "d\_res" represent the relative residual error (refer to SCS\footnote{\url{https://www.cvxgrp.org/scs/index.html}}, COSMO\footnote{\url{https://oxfordcontrol.github.io/COSMO.jl/stable/}}, and MOSEK\footnote{\url{https://docs.mosek.com/latest/faq/index.html}} for details on the stopping criteria settings).

\quad The numerical results in Tables \ref{small}, \ref{medium}, and \ref{large} reveal the following observations: (1) APADMM is the only algorithm capable of returning solutions within the given accuracy and the maximal iteration for problems of different scales. (2) In comparison with the second-order method MOSEK, while MOSEK quickly provides a solution for small-scale problems (as seen in Table \ref{small}), its computational time increases significantly with the problem size. For the medium-scale problem in Table \ref{medium}, APADMM outperforms MOSEK in terms of computation time. Furthermore, as indicated in Table \ref{large}, MOSEK encounters memory errors when attempting to solve large-scale problems. (3) In contrast to the first-order methods SCS and COSMO, APADMM exhibits superior performance in terms of iteration number and computation time. As seen in Table \ref{small}, COSMO fails to return a solution satisfying the given accuracy within the maximum iteration for small-scale problems, and SCS faces the same challenge for medium and large-scale problems.

\begin{table}[!htbp]
\centering
\caption{The performance of \textbf{MOSEK}, \textbf{COSMO}, \textbf{SCS}, and \textbf{APADMM} for solving small-scale ODC problems.}\label{small}
\begin{tabular}{cccccccc}
\toprule   
Size of LTI system               &n=7,m=4       &n=8,m=6           &n=9,m=6            &n=10,m=3           &n=15,m=2            &n=24,m=6         \\
Number of vertices               &M=5           &M=8               &M=8                &M=6                &M=5                 &M=8              \\
Size of variable $\xi$           &178           &354               &438                &382                &669                 &2758             \\
Number of constraints            &206           &393               &480                &421                &753                 &2865             \\
Nonzeos in $\mathcal{\tilde{A}}^{\ast}$  &1213          &3226              &4302               &2716               &8244                &57694    \\
      \bottomrule
      \textbf{MOSEK}\\
      \quad p\_res             &2.5e-10          &2.8e-12           &8.1e-11            &5.4e-10            &2.9e-14             &1.5e-13          \\
      \quad d\_res             &6.0e-10          &5.6e-09           &8.0e-11            &8.9e-09            &1.5e-07             &5.4e-09          \\
      \quad$\mathrm{\eta_{gap}}$&1.7e-13         &1.7e-15           &4.8e-14            &4.4e-13            &1.5e-18             &1.5e-16          \\
      \quad iter               &11               &15                &14                 &12                 &17                  &16               \\
      \quad p\_obj             &8.1533           &4.6586            &4.1229             &17.2072            &1.6019              &2.5349           \\
      \quad d\_obj             &8.1533           &4.6586            &4.1229             &17.2072            &1.6019              &2.5349           \\
      \quad total time (s)      &1.500e-02        &1.600e-02         &3.100e-02          &3.100-02           &4.600e-02           &1.570e-01        \\
      \quad status             &$\verb"solved"$  &$\verb"solved"$   &$\verb"solved"$    &$\verb"solved"$    &$\verb"solved"$     &$\verb"solved"$  \\ 
      \hline   
      \textbf{COSMO(direct)}\\               
      \quad p\_res             &3.898e+00        &6.307e+00         &4.149e+00          &2.023e+00          &3.124e+00           &1.303e+01       \\
      \quad d\_res             &4.521e-01        &1.378e-02         &7.739e-02          &1.256e-02          &7.281e-03           &7.496e-01       \\
      \quad iter               &25000            &25000             &25000              &25000              &25000               &25000            \\
      \quad total time (s)      &8.154e+00        &1.161e+01         &2.030e+01          &2.825e+01          &3.480e+01           &1.213e+02        \\
      \quad status             &$\verb"max_iter"$&$\verb"max_iter"$ &$\verb"max_iter"$  &$\verb"max_iter"$  &$\verb"max_iter"$   &$\verb"max_iter"$\\  
      \hline                                           
    \textbf{SCS(direct)}\\
      \quad p\_abs             &1.164e-03        &1.856e-03         &6.250e-04          &8.367e-04          &1.684e-04           &1.956e-04        \\
      \quad d\_abs             &8.722e-06        &8.023e-05         &1.435e-05          &1.342e-04          &2.168e-06           &1.715e-06        \\
      \quad$\mathrm{\eta_{gap}}$&2.903e-04       &1.232e-04         &5.084e-04          &4.368e-04          &2.593e-04           &3.522e-04        \\
      \quad iter               &1200             &1100              &3750               &2800               &3050                &9625             \\
      \quad objective          &8.1531           &4.6586            &4.1225             &17.2076            &1.6013              &2.5338           \\
      \quad total time (s)      &6.79e-02         &1.11e-01          &3.84e-01           &2.79e-01           &5.06e-01            &4.96e+00         \\
      \quad obj\_bias          &1.98e-04         &2.5e-05           &4.37e-04           &3.87e-04           &6.18e-04            &1.11e-03         \\
      \quad status             &$\verb"solved"$  &$\verb"solved"$   &$\verb"solved"$    &$\verb"solved"$    &$\verb"solved"$     &$\verb"solved"$  \\ 
      \hline   
    \textbf{APADMM}\\ 
      \quad p\_abs             &1.131e-03        &2.097e-04         &1.580e-04          &8.211e-05          &9.063e-04           &2.666e-04        \\
      \quad d\_abs             &3.994e-05        &7.380e-06         &1.996e-05          &5.227e-05          &3.549e-05           &3.155e-05        \\
      \quad$\mathrm{\eta_{gap}}$&1.344e-03       &8.472e-04        &9.207e-04          &3.491e-02          &2.143e-04           &4.630e-04        \\
      \quad iter               &220              &204               &250                &225                &223                 &356              \\
      \quad p\_obj             &8.1534           &4.6585            &4.1228             &17.2073            &1.6010              &2.5341           \\
      \quad d\_obj             &8.1548           &4.6599            &4.12188            &17.2108            &1.6008              &2.5336           \\
      \quad total time (s)      &1.51e-02         &2.02e-02          &3.20e-02           &3.37e-02           &6.08e-02            &2.50e-01         \\
      \quad obj\_bias          &1.33e-04         &3.40e-05          &1.40e-04           &6.00e-05           &8.48e-04            &8.44e-04         \\
      \quad status             &$\verb"solved"$  &$\verb"solved"$   &$\verb"solved"$    &$\verb"solved"$    &$\verb"solved"$     &$\verb"solved"$  \\
\bottomrule
\end{tabular}
\end{table}
\begin{table}[!htbp]
\centering
\caption{The performance of \textbf{MOSEK}, \textbf{COSMO}, \textbf{SCS}, and \textbf{APADMM} for solving medium-scale ODC problems.}\label{medium}
\begin{tabular}{cccccccc}
\toprule
Size of LTI system               &n=40,m=6           &n=60,m=5           &n=80,m=6           &n=100,m=8          &n=120,m=8           &n=150,m=8     \\     
Number of vertices               &M=8                &M=6                &M=8                &M=10               &M=8                 &M=10               \\
Size of variable $\xi$           &7421               &12653              &28974              &55533              &79698               &124094             \\
Number of constraints            &7641               &13125              &29661              &56386              &80856               &125811             \\
Nonzeos in $\mathcal{\tilde{A}}^{\ast}$  &241981             &571817             &1786744            &4367733            &7445888             &14325184           \\                             
      \bottomrule
      \textbf{MOSEK}\\
      \quad p\_res               &6.2e-13            &1.4e-14            &1.5e-13            &3.1e-14            &7.8e-15             &7.4e-14            \\
      \quad d\_res               &6.9e-09            &2.8e-08            &8.08e-08           &1.0e-08            &3.1e-09             &1.7e-06            \\
      \quad $\mathrm{\eta_{gap}}$ &1.8e-15            &8.5e-18           &1.5e-15            &1.5e-16            &4.7e-17             &1.7e-17            \\
      \quad iter                 &15                 &16                 &19                 &17                 &20                  &22                 \\
      \quad p\_obj               &48.2255            &6.1406             &1.6705             &48.4216            &31.2284             &3.87385            \\
      \quad d\_obj               &48.2255            &6.1406             &1.6705             &48.4216            &31.2284             &3.87385            \\
      \quad total time (s)        &8.430e-01          &3.016e+00          &1.550e+01          &6.242e+01          &2.568e+02           &4.436e+02          \\
      \quad status               &$\verb"solved"$    &$\verb"solved"$    &$\verb"solved"$    &$\verb"solved"$    &$\verb"solved"$     &$\verb"solved"$             \\\hline
      \textbf{COSMO(direct)}\\              
      \quad p\_res               &2.654e+01          &2.096e+02          &3.625e+01          &9.763e+01          &1.395e+02           &1.153e+02        \\
      \quad d\_res               &6.557e-02          &3.095e-02          &5.057e-02          &3.114e-01          &2.108e+00           &1.301e-01        \\
      \quad iter                 &25000              &25000              &25000              &25000              &25000               &25000             \\
      \quad total time (s)        &3.214e+02          &7.453e+02          &1.654e+03          &3.845e+03          &7.138e+03           &9.694e+03         \\
      \quad status               &$\verb"max_iter"$  &$\verb"max_iter"$  &$\verb"max_iter"$  &$\verb"max_iter"$  &$\verb"max_iter"$   &$\verb"max_iter"$         \\ \hline                                            
    \textbf{SCS(direct)}\\
      \quad p\_abs               &1.206e-03          &6.402e-04          &3.157e-03          &1.012e-02          &3.462e-03           &5.152e-03         \\
      \quad d\_abs               &7.394e-05          &3.510e-06          &8.599e-06          &4.450e-05          &3.265e-06           &1.578e-05         \\
      \quad $\mathrm{\eta_{gap}}$&4.483e-03          &2.727e-03          &1.756e-02          &2.171e-01          &4.730e-02           &1.388e-01         \\
      \quad iter                 &10225              &25000              &25000              &25000              &25000               &25000             \\
      \quad objective            &48.2183            &6.1334             &1.6381             &47.8352            &31.0848             &3.6930            \\
      \quad total time (s)        &1.78e+01           &1.13e+02           &3.27e+02           &7.10e+02           &1.22e+03            &2.50e+03          \\
      \quad obj\_bias            &7.19e-03           &7.15e-03           &3.24e-02           &5.86e-01           &1.44e-01            &1.81e-01          \\
      \quad status               &$\verb"solved"$    &$\verb"max_iter"$  &$\verb"max_iter"$  &$\verb"max_iter"$  &$\verb"max_iter"$   &$\verb"max_iter"$ \\ \hline  
    \textbf{APADMM}\\
      \quad p\_abs               &8.657e-04          &2.106e-03          &3.688e-03          &3.125e-03          &3.396e-03           &3.418e-03         \\
      \quad d\_abs               &1.380e-05          &3.738e-05          &1.342e-05          &3.961e-06          &1.422e-06           &2.815e-06         \\
      \quad $\mathrm{\eta_{gap}}$&9.705e-03          &8.180e-04          &3.963e-04          &9.708e-03          &6.060e-03           &1.495e-04         \\
      \quad iter                 &816                &433                &418                &1138               &1075                &793               \\
      \quad p\_obj               &48.2258            &6.1354             &1.6591             &48.4200            &31.2178             &3.8623            \\
      \quad d\_obj               &48.2161            &6.1346             &1.6586             &48.4103            &31.2118             &3.8624            \\
      \quad total time (s)        &1.68e+00           &2.77e+00           &6.75e+00           &4.31e+01           &7.15e+01            &1.09e+02          \\
      \quad obj\_bias            &3.20e-04           &5.19e-03           &1.15e-02           &1.57e-03           &1.06e-02            &1.16e-02          \\
      \quad status               &$\verb"solved"$    &$\verb"solved"$    &$\verb"solved"$    &$\verb"solved"$    &$\verb"solved"$     &$\verb"solved"$          \\
\bottomrule
\end{tabular}
\end{table}

\begin{table}[!htbp]
\centering
\caption{The performance of \textbf{MOSEK}, \textbf{SCS}, and \textbf{APADMM} for solving large-scale ODC problems.}\label{large}
\begin{tabular}{cccccccc}
\toprule
Size of LTI system              &n=300, m=4          &n=350,m=4         &n=400, m=5          &n=480, m=3                              \\
Number of vertices              &M=6                 &M=6               &M=6                 &M=4                                          \\
Size of variable $\xi$          &305515              &415336            &546696              &539463                                       \\
Number of constraints           &317260              &431385            &563415              &578646                                       \\
Nonzeos in $\mathcal{\tilde{A}}^{\ast}$ &65652859    &104442568         &155504316           &178508123                                    \\
\bottomrule
    \textbf{MOSEK}\\ 
      \quad p\_res              &8.6e-14             &1.1e-13           &1.2e-14             &***                                          \\
      \quad d\_res              &5.1e-07             &4.4e-06           &2.0e-07             &***                                          \\
      \quad $\mathrm{\eta_{gap}}$&9.2e-15            &4.1e-16           &4.0e-15             &***                                          \\
      \quad iter                &21                  &24                &21                  &***                                          \\
      \quad p\_obj              &38.1333             &6.8503            &103.9257            &***                                          \\
      \quad d\_obj              &38.1333             &6.8503            &103.9257            &***                                          \\
      \quad total time (s)       &1.385e+04           &3.574e+04         &1.619e+05           &***                                          \\
      \quad status              &$\verb"solved"$     &$\verb"solved"$   &$\verb"solved"$     &$\verb"memory error"$         \\
      \hline                                     
    \textbf{SCS(direct)}\\
      \quad p\_abs              &6.748e-03           &2.640e-03         &1.053e-02           &5.524e-03                                    \\
      \quad d\_abs              &1.410e-05           &5.786e-06         &2.646e-05           &9.562e-06                                    \\
      \quad$\mathrm{\eta_{gap}}$ &3.904e-01          &1.345e-01         &5.251e-01           &2.594e-01                                    \\
      \quad iter                &10000               &10000             &10000               &10000                                        \\
      \quad objective           &37.5071             &6.5015            &102.6840            &57.3878                                     \\
      \quad total time (s)       &9.27e+03            &1.67e+04          &2.84e+04            &5.94e+04                                     \\
      \quad obj\_bias           &6.26e-01            &3.49e-01          &1.24e-01            &-                                            \\
      \quad status              &$\verb"max_iter"$   &$\verb"max_iter"$ &$\verb"max_iter"$   &$\verb"max_iter"$          \\
      \hline
    \textbf{APADMM}\\
      \quad p\_abs              &1.090e-02           &6.424e-03         &1.042e-02           &3.216e-02                                    \\
      \quad d\_abs              &4.141e-06           &6.578e-06         &4.701e-06           &7.001e-05                                    \\
      \quad$\mathrm{\eta_{gap}}$&7.543e-03           &6.021e-04         &1.953e-02           &3.092e-03                                    \\
      \quad iter                &1362                &1413              &1683                &1150                                         \\
      \quad p\_obj              &38.0940             &6.7651            &103.9076            &57.8089                                     \\
      \quad d\_obj              &38.0865             &6.7651            &103.8881            &57.8120                                     \\
      \quad total time (s)       &1.54e+03            &2.91e+03          &5.87e+03            &9.18e+03                                     \\
      \quad obj\_bias           &3.93e-02            &8.52e-02          &1.57e-02            &-                                            \\
      \quad status              &$\verb"solved"$     &$\verb"solved"$   &$\verb"solved"$     &$\verb"solved"$                              \\
\bottomrule  
\end{tabular}
\footnotetext{Entries marked *** indicate failure due to memory limitations.}
\end{table}

\section{Conclusions and future work}\label{sec6}

In this paper, we developed an accelerated PADMM to efficiently solve the CP problem, which is the relaxation of the $\mathcal{H}_{2}$-guaranteed cost ODC problem.
Establishing the equivalence between the (generalized) PADMM and the relaxed PPA allows us to accelerate the (generalized) PADMM by using the Halpern fixed-point iteration method, achieving a fast $\mathcal{O}(1/k)$ convergence rate. To enhance the efficiency of Algorithm \ref{alg4}, we employed the lifting technique for solving the PADMM subproblems. Numerical experiments on medium and large-scale CP problems, derived from $\mathcal{H}_{2}$-guaranteed cost ODC problems, demonstrated that the proposed accelerated PADMM outperforms COSMO, MOSEK, and SCS in terms of computation time. Finally, it is important to note that the projections in our algorithm are independent, and we used direct methods for solving the linear systems. In future work, we will explore the development of a parallel framework for solving the non-smooth term and incorporate iterative methods for solving the linear systems into our algorithm.

\bibliography{reference}
\bibliographystyle{unsrt}

\newpage
\appendix

\section{The definition of $V_{j1}$ and $V_{j2}$}\label{appxa}

According the definition of the set $\mathcal{W}$ in \eqref{Wset}, $W_{1}$ and $W_{2}$ should be $m$ block-diagonal matrices with $\sum_{i=1}^{m}D_{i}=n$. The coefficient matrix for the equality constraints \eqref{eqvij} can be constructed by utilizing the symmetry of $W_1$ and the sparse structure of $W_2$. Therefore, we define the matrices $\left\{V_i\right\}_{i=1}^{2m}$ to extract the off-diagonal block parts from $W$, and restrict the corresponding positions are zeros. Specifically, for $1\leq i\leq m-1,$ we define matrices $V_{i}$ as follows:
\begin{align}
&V_{1}=\left[
        \begin{array}{c}
          e_{1}^{\top} \\
          \vdots \\
          e_{D_{1}}^{\top} \\
        \end{array}
      \right],~
V_{2}=\left[
         \begin{array}{c}
                e_{D_{1}+1}^{\top} \\
                \vdots \\
                e_{D_{1}+D_{2}}^{\top}\notag\\
              \end{array}
            \right],\ldots,
V_{m-1} =\left[
\begin{array}{c}
e_{D_{1}+\ldots D_{m-2}+1}^{\top}\\
\vdots \\
e_{D_{1}+\ldots D_{m-2}+D_{m-1}}^{\top}\\
\end{array}
\right]\in\mathbb{R}^{D_{m-1}\times p}\notag.
\end{align}
For $i=m,$ the matrix $V_{m}$ is given by
\begin{align}
V_{m}=e_{D_{1}+D_{2}+\ldots+D_{m}}^{\top}=e_{n}^{\top}\in\mathbb{R}^{1\times p}\notag.
\end{align}
For $k=1,\ldots,m,$ $V_{m+k}$ is denoted by the $(n+k)$-th element is $1$ and all other components are $0$, i.e.,
\begin{align}
V_{m+k}=[\underbrace{O_{1\times D_{1}},\ldots,O_{1\times D_{m}}}_{\sum_{i=1}^{m}D_{i}=n},\underbrace{0,\ldots,1}_{n+k},\ldots,0]\in \mathbb{R}^{1\times{p}}\notag.
\end{align}
The block diagonal structure of $W_{1}$ and $W_{2}$ in \eqref{Wset} and the symmetry of $W_{1}$ imply that $N=(3m(m-1))/2$ off-diagonal block submatrices of $W$ need to be zero. Using the defined matrices $\left\{V_i\right\}_{i=1}^{2m}$, we extract the aforementioned off-diagonal blocks and set them to be zero through the following linear constraints:
\begin{align}
&V_{i}WV_{i+1}^{\top}=0,\ldots,V_{i}WV_{i+m-1}^{\top}=0, \notag\\
&V_{i}WV_{i+m+1}^{\top}=0,\ldots,V_{i}WV_{2m}^{\top}=0,\, i=1,\ldots,m-1,\notag
\end{align}
and
\begin{equation*}
V_{m}WV_{m+1}^{\top}=0,\ldots,V_{m}WV_{2m-1}^{\top}=0, \, i=m.
\end{equation*}
{Next, we divide the index array $[1,2,\ldots,\frac{3m(m-1)}{2}]$ into $m$ blocks:
\begin{equation*}
\left[\begin{array}{ccc}
\underbrace{1,2,\ldots}_{1}|\ldots|\underbrace{\ldots}_{i}|\ldots|\underbrace{\ldots,\frac{3m(m-1)}{2}}_{m} \\
\end{array}\right],
\end{equation*}
where for $i$-th block, the starting index is
\begin{align*}
(i-1)(2m-2)-\frac{(i-1)(i-2)}{2}+1,~i=1, \ldots, m.
\end{align*}

According to the above decomposition, we are able to reorder $\left\{V_i\right\}_{i=1}^{2m}$, and construct $V_{j1}$ and $V_{j2}$ for $j=1,\ldots, N$.
Specifically, we define the matrix $V_{j1}$ as follows
\begin{equation*}
V_{j1}=\left\{
\begin{array}{ll}
V_{1},&\mathrm{for}~j=1, 2, \ldots, 2m-2,\\
V_{2},&\mathrm{for}~j=2m-1, \ldots, 4m-5,\\
\vdots&\\
V_{m-1},&\mathrm{for}~j=T+1,~\ldots,~T+m,\\
V_{m},&\mathrm{for}~j=L+1,~\ldots,~L+m-1,
\end{array}
\right.
\end{equation*}
where
\begin{align*}
T&=(m-2)(2m-2)-\frac{(m-2)(m-3)}{2},\\
L&=(m-1)(2m-2)-\frac{(m-1)(m-2)}{2}.
\end{align*}
For each $s(i)=[i+1,~\ldots,~i+(m-1),~i+(m+1),~\ldots,~2m],$ the matrix $V_{j2}$ is defined by
\begin{equation*}
V_{j2}=\left\{
\begin{array}{ll}
V_{s(1)_{j}},  &\mathrm{for}~j=1,~2, \ldots,~2m-2,\\
V_{s(2)_{j-(2m-2)}}, &\mathrm{for}~j=2m-1,~\ldots,~4m-5,\\
\vdots &\\
V_{s(m-1)_{j-T}},  &\mathrm{for}~j=T+1,~\ldots,~T+m,\\
V_{s(m)_{j-L}}, &\mathrm{for}~j=L+1,~\ldots,~L+m-1.
\end{array}
\right.
\end{equation*}} 

\section{The equivalence of the semi-proximal ADMM and the (partial) PPA}\label{appxb}

In this section, we establish an equivalence between the proximal point method and the semi-proximal ADMM, which covers the version of PADMM used in this paper. Consider the following linearly constrained convex optimization problem with separable objective functions:
\begin{equation}
	\begin{array}{ll}
		\underset{x,y}{\mathrm{\min}} &  f(x) + g(y)  \\[5pt]
		\textrm{s.t.}
		& \cF x + \cG y = c,
	\end{array}
	\label{eq-p}
\end{equation}
where $f:\cX\to (-\infty,+\infty]$ and $g:\cY\to (-\infty,+\infty]$ are two proper closed convex functions, $\cF :\cX \to \cZ$ and $\cG:\cX \to\cZ$ are two given linear operators,
$c\in\cZ$ is given data, and $\cX,\cY$ and $\cZ$ are all real finite-dimensional Euclidean spaces, each equipped with an inner product $\inprod{\cdot}{\cdot}$ and its induced norm $\norm{\cdot}.$ The dual problem of \eqref{eq-p} can be written as 
\begin{equation}
	\label{eq-d}
\underset{z}{\mathrm{\max}} \,\{-\inprod{c}{z} - f^*(-\cF^*z) - g^*(-\cG^*z)\}.
\end{equation}
The KKT system associated with \eqref{eq-p} and \eqref{eq-d} is given by
\begin{equation}\label{KKT}
	\left\{
	\begin{aligned}
		0 \in{}&\partial f(x) + \cF^*z, \\[5pt]
		0 \in{}&\partial g(y) + \cG^*z, \\[5pt]
		0 ={}& c - \cF x - \cG y.
	\end{aligned}
	\right.
\end{equation}
Let $\Omega$ be the solution set to the above KKT system. Suppose that $\Omega \not = \emptyset$.
Denote
\[\cT_l (x,y,z) := \left(
\begin{array}{c}
	\partial f(x) +  \cF^*z\\
	\partial g(y) + \cG^*z \\
	c - \cF x- \cG y  \\
\end{array}
\right),
\]
then $\cT_l$ is a maximal monotone operator and $\cT_l^{-1}(0) \not = \emptyset$. Given $\sigma >0$, the augmented Lagrangian function associated with \eqref{eq-p} is given as follows:
\begin{equation*}
	\cL_{\sigma}(x,y;z) = f(x) + g(y)
	+\inprod{z}{\cF x +\cG y -c}
	+ \frac{\sigma}{2}\norm{\cF x +\cG y-c}^2.
	\label{ALM}
\end{equation*}
Given two self-adjoint linear operators $\cS_1:\cX\to\cX$ and $\cS_2:\cY\to\cY$, define a self-adjoint linear operator $\cS:\cX\times\cY\times\cZ\to \cX\times\cY\times\cZ$ as follows
\[\cS := \left(
\begin{array}{ccc}
	\cS_1 &  &  \\
	& \sigma\cG^*\cG+\cS_2 & \cG^*  \\
	& \cG &  \sigma^{-1}\cI \\
\end{array}
\right),
\]
where $\cI$ is the identity operator on $\cZ$. The following proposition establishes the equivalence between the proximal point method and the semi-proximal ADMM.

\begin{prop}\label{prop-B}
	Let $\cS_1$ and $\cS_2$ be two self-adjoint positive semi-definite linear operators such that $\partial f(\cdot)+\sigma\cF^*\cF+\cS_1$ and $\partial g(\cdot)+\sigma\cG^*\cG+\cS_2$ are maximal and strongly monotone, respectively. Consider $(x^0,y^0,z^0)\in\operatorname{dom}(f)\times\operatorname{dom}(g)\times\cZ$ and $\rho\in \mathbb{R}$. Then the point $(x^+,y^+,z^+)$ generated by the following generalized semi-proximal ADMM scheme \eqref{prox-ADMM}
	\begin{equation}
		\label{prox-ADMM}
		\left\{
		\begin{aligned}
			\bx^{+} ={}& \underset{x}{\mathrm{argmin}}~ \cL_{\sigma}(x,y^0;z^0) + \frac{1}{2}\norm{x - x^0}_{\cS_1}^2, \\[5pt]
			\bz^{+} ={}& z^0 + \sigma(\cF \bx^{+} + \cG y^0 - c),\\[5pt]
			\by^{+} ={}& \underset{y}{\mathrm{argmin}}~ \cL_{\sigma}(\bx^{+},y;\bz^{+}) + \frac{1}{2}\norm{y - y^0}_{\cS_2}^2, \\[5pt]
			&\hspace{-1.0cm}(x^{+},y^{+},z^{+}) = (1-\rho)(x^{0},y^{0},z^{0})
			+ \rho (\bx^{+},\by^{+},\bz^{+})
		\end{aligned}
		\right.
	\end{equation}
	is equivalent to the one generated by the following proximal point scheme $\eqref{ppa-2}$
	\begin{equation}\label{ppa-2}
		\left\{\begin{aligned}
			&0\in\cT_l(\bx^{+},\by^{+},\bz^{+}) + \cS(\bx^{+}-x^0,\by^{+}-y^0,\bz^{+}-z^0),
			\\[5pt]
			&(x^{+},y^{+},z^{+}) = (1-\rho)(x^{0},y^{0},z^{0})
			+ \rho (\bx^{+},\by^{+},\bz^{+}).
		\end{aligned}
		\right.
	\end{equation}
\end{prop}

\begin{proof}
	Since $\partial f(\cdot)+\sigma\cF^*\cF+\cS_1$ and $\partial g(\cdot)+\sigma \cG^*\cG+\cS_2$ are assumed to be maximal and strongly monotone, the objective function of each subproblem in the semi-proximal ADMM scheme is a proper closed strongly convex function \cite[Exercise 8.8 and Exercise 12.59]{RW1998}. Consequently, one can deduce from \cite[Theorem 8.15 and Proposition 12.54]{RW1998} that $\bx^{+} = {}\underset{x}{\mathrm{argmin}}~ \cL_{\sigma}(x,y^0;z^0) + \frac{1}{2}\norm{x - x^0}_{\cS_1}^2$ if and only if 
	\begin{equation*}	
		0\in\partial f(\bx^{+}) + \cF^*(z^0 + \sigma(\cF\bx^{+} +  \cG y^{0} -c)) + \cS_1(\bx^{+} - x^0),	
	\end{equation*}
	and 
	$\by^{+} = {} \underset{y}{\mathrm{argmin}}~\cL_{\sigma}(\bx^{+},y;\bz^{+}) + \frac{1}{2}\norm{y - y^0}_{\cS_2}^2$
	if and only if 
	\begin{equation*}	
		0\in\partial g(\by^{+}) + \cG^*(\bz^{+}  + \sigma(\cF \bx^{+} +  \cG\by^{+} -c)) + \cS_2(\by^{+} - y^0).	
	\end{equation*}
	Hence, \eqref{prox-ADMM} is equivalent to 	
	\begin{equation}\label{ppa-admm-2}
		\left\{ \begin{aligned}
			& 0\in\partial f(\bx^{+}) + \cF^*(z^0 + \sigma(\cF\bx^{+} +  \cG y^{0} -c)) + \cS_1(\bx^{+} - x^0), \\[5pt]
			& \bz^{+} = z^0 + \sigma(\cF\bx^{+} +  \cG y^{0} -c),\\[5pt]
			& 0\in\partial g(\by^{+}) + \cG^*(\bz^{+}  + \sigma(\cF \bx^{+} +  \cG\by^{+} -c)) + \cS_2(\by^{+} - y^0),
			\\[5pt]
			&(x^{+},y^{+},z^{+}) = (1-\rho)(x^{0},y^{0},z^{0})
			+ \rho (\bx^{+},\by^{+},\bz^{+}).
		\end{aligned}\right.
	\end{equation}
	Note that for any $(x^{+},y^{+},z^{+})$ satisfying \eqref{ppa-admm-2}, one has 
	$$-\sigma^{-1}(\bz^{+} - z^0)= -(\cF\bx^{+} +  \cG y^{0} -c)= c - \cF\bx^{+}- \cG\by^{+} + \cG(\by^{+} - y^0).$$
	Therefore, \eqref{ppa-admm-2} can be recast as
	\begin{equation*}
		\left\{ \begin{aligned}
			& 0\in\partial f(\bx^{+}) + \cF^*\bz^{+} + \cS_1(\bx^{+} - x^0), \\[5pt]
			& 0\in\partial g(\by^{+}) + \cG^*\bz^{+} + (\sigma\cG^*\cG + \cS_2)(\by^{+} - y^0) + \cG^*(\bz^{+} - z^0), \\[5pt]
			& 0 = c - \cF\bx^{+}- \cG\by^{+} + \cG(\by^{+} - y^0) +
			\sigma^{-1}(\bz^{+} - z^0),
			\\[5pt]
			&(x^{+},y^{+},z^{+}) = (1-\rho)(x^{0},y^{0},z^{0})
			+ \rho (\bx^{+},\by^{+},\bz^{+}),
		\end{aligned}\right. 
	\end{equation*}
which is exactly the scheme \eqref{ppa-2} by the definition of $\cT_l$ and $\cS$. Hence, the point $(x^+,y^+,z^+)$ generated by the generalized semi-proximal ADMM scheme \eqref{prox-ADMM} is equivalent to the one generated by the proximal point scheme $\eqref{ppa-2}$. 
\end{proof}

\end{document}